\documentclass[
	paper=a4,%
	pagesize,%
	numbers=noendperiod,%
	captions=nooneline,%
	abstracton%
]{scrartcl}

\newcommand*{\mytitle}{Likelihood\ inference\ for\ Archimedean\ copulas} 
\newcommand*{\myauthorone}{Marius\ Hofert}
\newcommand*{\mycontactone}{RiskLab,\ Department\ of\ Mathematics,\ ETH\ Zurich,\ 8092\ Zurich,\ Switzerland,\ \href{mailto:marius.hofert@math.ethz.ch}{\nolinkurl{marius.hofert@math.ethz.ch}}}
\newcommand*{\mycomment}{The\ author\ (Willis\ Research\ Fellow)\ thanks\ Willis\ Re\ for\ financial\ support\ while\ this\ work\ was\ being\ completed.}
\newcommand*{\myauthortwo}{Martin\ M\"achler}
\newcommand*{\mycontacttwo}{Seminar\ f\"ur\ Statistik,\ ETH\ Zurich,\ 8092\ Zurich,\ Switzerland,\ \href{mailto:maechler@stat.math.ethz.ch}{\nolinkurl{maechler@stat.math.ethz.ch}}}
\newcommand*{\myauthorthree}{Alexander\ J.\ McNeil}
\newcommand*{\mycontactthree}{Department\ of\ Actuarial\ Mathematics\ and\ Statistics,\ Heriot-Watt\ University,\ Edinburgh,\ EH14 4AS,\ Scotland,\ \href{mailto:A.J.McNeil@hw.ac.uk}{\nolinkurl{A.J.McNeil@hw.ac.uk}}}
\newcommand*{\mysubject}{Article}

\usepackage[T1]{fontenc}%
\usepackage{lmodern}%
\usepackage[american]{babel}%
\usepackage{microtype}%

\usepackage[nouppercase]{scrpage2}%

\usepackage{eso-pic}%
\usepackage{rotating}%
\usepackage{amsmath}%
\usepackage{mathtools}%
\usepackage{amssymb}%
\usepackage{amsthm}%
\usepackage{bm}%
\usepackage{bbm}%
\usepackage{enumitem}%
\usepackage{graphicx}%
\usepackage{grffile}%
\usepackage{tikz}%
\usepackage{wrapfig}%
\usepackage{tabularx}%
\usepackage{dcolumn}%
\usepackage{booktabs}%
\usepackage{multirow}%
\usepackage{vruler}%
\usepackage[round]{natbib}%
\usepackage[
	hypertexnames=false,%
	setpagesize=false,%
	pdfborder={0 0 0},%
   pdfstartview=Fit,%
	bookmarksopen=true,%
	bookmarksnumbered=true%
]{hyperref}%
\hypersetup{
  	pdfauthor=\myauthorone,%
	pdftitle=\mytitle,%
	pdfsubject=\mysubject%
}
\setkomafont{pageheadfoot}{\normalfont\normalcolor\sffamily}%
\automark{section}%
\setkomafont{pagenumber}{\normalfont\normalcolor\sffamily}%
\setkomafont{captionlabel}{\normalfont\normalcolor\sffamily\bfseries}%

\makeatletter
\newcommand\myisodate{\number\year-\ifcase\month\or 01\or 02\or 03\or 04\or 05\or 06\or 07\or 08\or 09\or 10\or 11\or 12\fi-\ifcase\day\or 01\or 02\or 03\or 04\or 05\or 06\or 07\or 08\or 09\or 10\or 11\or 12\or 13\or 14\or 15\or 16\or 17\or 18\or 19\or 20\or 21\or 22\or 23\or 24\or 25\or 26\or 27\or 28\or 29\or 30\or 31\fi}%
\makeatother
\newcolumntype{d}[2]{D{.}{.}{#1.#2}}%
\newcommand*{\abstractnoindent}{}%
\let\abstractnoindent\abstract
\renewcommand*{\abstract}{\let\quotation\quote\let\endquotation\endquote
\abstractnoindent}
\deffootnote[1em]{1em}{1em}{\textsuperscript{\thefootnotemark}}%

\newcommand*{\textcite}[2][]{\citet[#1]{#2}}
\newtheoremstyle{mythmstyle}%
	{0.5em}%
	{0.5em}%
	{}%
	{}%
	{\sffamily\bfseries}%
	{}%
	{\newline}%
	{\thmname{#1}\ \thmnumber{#2}\ \thmnote{(#3)}}%
\newcommand*{\myskip}{~\vspace{-1.2em}}%
\newcommand*{\myskipalgo}{~\vspace{-2.2em}}%
\theoremstyle{mythmstyle}
\newtheorem{definition}{Definition}[section]%
\newtheorem{proposition}[definition]{Proposition}

\newtheorem{theorem}[definition]{Theorem}
\newtheorem{corollary}[definition]{Corollary}
\newtheorem{remark}[definition]{Remark}
\newtheorem{example}[definition]{Example}
\newtheorem{algorithm}[definition]{Algorithm}
\renewcommand*\proofname{Proof}
\makeatletter%
\renewenvironment{proof}[1][\proofname]{\par
  \pushQED{\qed}%
  \normalfont\topsep2\p@\@plus2\p@\relax
  \trivlist
  \item[\hskip\labelsep
	  \sffamily\bfseries #1]\mbox{}\hfill\\*\ignorespaces
}{%
  \popQED\endtrivlist\@endpefalse
}
\makeatother

\newlength{\mytopsep}
\newlength{\myitemsep}
\newcommand*{\eps}{\varepsilon}
\newcommand{\T}{\ensuremath{^\mathsf{T}}\hspace{-0.5mm}}
\newcommand*{\omu}[3]{\underset{#3}{\overset{#1}{#2}}}
\renewcommand*{\vec}[2]{{\biggl(\begin{matrix} #1 \\ #2 \end{matrix}\biggr)}}
\newcommand*{\I}{\mathbbm{1}}
\newcommand*{\IN}{\mathbb{N}}

\newcommand*{\IR}{\mathbb{R}}

\newcommand*{\IE}{\mathbb{E}}

\renewcommand*{\S}{\operatorname*{S}}

\newcommand*{\argsup}{\operatorname*{argsup}}

\newcommand*{\Li}{\operatorname*{Li}}
\renewcommand*{\th}{\bm{\theta}}
\newcommand*{\psii}{{\psi^{-1}}}
\newcommand*{\psis}[2]{{\psi_{#1}^{#2}}}
\newcommand*{\psiis}[1]{{\psi_{#1}^{-1}}}

\newcommand*{\LS}{\mathcal{LS}}
\newcommand*{\LSi}{\LS^{-1}}

\newcommand*{\Geo}{\operatorname*{Geo}}

\newcommand*{\Sibuya}{\operatorname*{Sibuya}}
\newcommand*{\GIG}{\operatorname*{GIG}}

\newcommand*{\Log}{\operatorname*{Log}}

\newcommand*{\var}[3]{{{#1}^{\text{#2}}_{#3}}}

\begin{document}
\AddToShipoutPicture{%
  \begin{tikzpicture}[remember picture, overlay]
		\node[scale=10,rotate=54.74,text opacity=0.2] at (current page.center){\normalfont\sffamily Submitted};
  \end{tikzpicture}%
}
\setvruler[10pt][1][1][4][1][0pt][0pt][-28pt][\textheight]%
\thispagestyle{plain}
	\begin{center}
		\sffamily
		{\bfseries\LARGE\mytitle\par}
		\bigskip
		{\Large\myauthorone\footnote{\mycontactone. \mycomment},\ \myauthortwo\footnote{\mycontacttwo},\ \myauthorthree\footnote{\mycontactthree}\par
	    \bigskip
	    \myisodate\par}
	\end{center}
	\par\bigskip
	\begin{abstract}
		Explicit functional forms for the generator derivatives of well-known one-parameter Archimedean copulas are derived. These derivatives are essential for likelihood inference as they appear in the copula density, conditional distribution functions, or the Kendall distribution function. They are also required for several asymmetric extensions of Archimedean copulas such as Khoudraji-transformed Archimedean copulas. Access to the generator derivatives makes maximum-likelihood estimation for Archimedean copulas feasible in terms of both precision and run time, even in large dimensions. It is shown by simulation that the root mean squared error is decreasing in the dimension. This decrease is of the same order as the decrease in sample size. Furthermore, confidence intervals for the parameter vector are derived. Moreover, extensions to multi-parameter Archimedean families are given. All presented methods are implemented in the open-source \textsf{R} package \texttt{nacopula} and can thus easily be accessed and studied. 		
	\end{abstract}
	\minisec{Keywords}
		Archimedean copulas, maximum-likelihood estimation, confidence intervals, multi-parameter families.
	\minisec{MSC2010} 
	62H12, 62F10, 62H99, 65C60.%
\section{Introduction}\label{sec.intro}
   The well-known class of \textit{Archimedean copulas} consists of copulas of the form 
   \begin{align*}
		C(\bm{u})=\psi(\psi(u_1)+\dots+\psi(u_d)),\ \bm{u}\in[0,1]^d,
	\end{align*}
	with \textit{generator} $\psi$. In practical applications, $\psi$ belongs to a parametric family $(\psis{\th}{})_{\th\in\Theta}$ whose parameter vector $\th$ needs to be estimated. 
	\par
	There are several known approaches for estimating parametric Archimedean copula families; see \textcite{hofertmaechlermcneil2011a} for an overview and a comparison of some estimators. 
	\par
	In the work at hand, we consider a (semi-)parametric estimation approach based on the likelihood. There are two significant obstacles to overcome. The first one is to derive tractable algebraic expressions for the generator derivatives and thus the copula density. The second is to evaluate these expressions efficiently in terms of both precision and run time. 
	\par
	Although the density of an Archimedean copula has an explicit form in theory, accessing the required derivatives is known to be challenging, especially in large dimensions. For example, \textcite{bergaas2009} mention that for Archimedean copulas it is not straightforward to derive the density in general for all parametric families. For the Gumbel family, they say that one has to resort to a computer algebra system, such as Mathematica or the function \texttt{D} in \textsf{R}, to derive the $d$-dimensional density. Note that computations based on computer algebra systems often fail already in low dimensions. Even if a theoretical formula can be computed, the numerical evaluation of such (typically lengthy) formulas is prone to errors since they are not given in a numerically tractable form. This often requires to work with a large number of significant digits which is typically far too slow to be applied in large-scale simulation studies (for example, to access the quality of goodness-of-fit testing procedures). Furthermore, as we will point out below, results obtained by computer algebra systems can be unreliable. 
	\par
	Generator derivatives for some important Archimedean families can be found in \textcite{shi1995}, \textcite{barbegenestghoudiremillard1996}, and \textcite{wuvaldezsherris2007}, however, in recursive form. In this work, we derive explicit formulas for the generator derivatives of well-known Archimedean families in any dimension. These derivatives are interesting in their own right, for example, for accessing densities, for building conditional distribution functions, or for evaluating the Kendall distribution function. They can also be used to explicitly compute densities of asymmetric extensions of Archimedean copulas such as Khoudraji-transformed Archimedean copulas. 
	\par
	We then tackle the problem of maximum-likelihood estimation for Archimedean copulas for these families. Focus is put on large, say ten to one hundred, dimensions since they are the most relevant in practice; see \textcite{embrechtshofert2011c}. Note that the considered Gumbel family is also an extreme value copula, for which densities in general are rarely known. \textcite{hofertmaechlermcneil2011a} show the excellent performance of the maximum-likelihood estimator as measured by both precision and run time in a large-scale comparison with various other estimators up to dimension one hundred. Furthermore, to add transparency, all the algorithms used in this paper are implemented in the open source \textsf{R} package \texttt{nacopula}, so that the interested reader can study the non-trivial details of the numerical implementation and the numerous tests conducted in more detail. In the work at hand, we also consider examples of multi-parameter Archimedean families. In contrast to method-of-moments-like estimation procedures such as the one based on Kendall's tau, maximum-likelihood estimation is not limited to the one-parameter case. Furthermore, we address the problem of computing initial intervals for the optimization of the log-likelihood for the multi-parameter Archimedean families considered. Additionally, we show how confidence intervals for the copula parameter vector can be constructed. 
	\par
	The paper is organized as follows. In Section \ref{sec.ac}, we briefly recall the notion of Archimedean copulas and the families considered. Section \ref{sec.est} presents explicit functional forms of the generator derivatives of these families and the corresponding copula densities are derived. In Section \ref{sec.nvsd}, the root mean squared error is investigated as a function of the dimension. Section \ref{sec.ci} presents methods for constructing confidence intervals for the copula parameter vector. In Section \ref{sec.multi.param} we address extensions to multi-parameter Archimedean families, including a strategy for computing initial intervals and two examples of two-parameter families. Finally, Section \ref{sec.con} concludes.
\section{Archimedean copulas}\label{sec.ac}
	\begin{definition}
		An \textit{(Archimedean) generator} is a continuous, decreasing function $\psi:[0,\infty]\to[0,1]$ which satisfies $\psi(0)=1$, $\psi(\infty)=\lim_{t\to\infty}\psi(t)=0$, and which is strictly decreasing on $[0,\inf\{t:\psi(t)=0\}]$. A $d$-dimensional copula $C$ is called \textit{Archimedean} if it permits the representation 
		\begin{align}
			C(\bm{u})=\psi(\psii(u_1)+\dots+\psii(u_d)),\ \bm{u}\in[0,1]^d,\label{C}
		\end{align}
		for some generator $\psi$ with inverse $\psii:[0,1]\to[0,\infty]$, where $\psii(0)=\inf\{t:\psi(t)=0\}$.
	\end{definition}
	\par
	\textcite{mcneilneslehova2009} show that a generator defines an Archimedean copula if and only if $\psi$ is $d$\textit{-monotone}, meaning that $\psi$ is continuous on $[0,\infty]$, admits derivatives up to the order $d-2$ satisfying $\smash[t]{(-1)^k\frac{d^k}{dt^k}\psi(t)\ge 0}$ for all $k\in\{0,\dots,d-2\}$, $t\in(0,\infty)$, and $\smash[t]{(-1)^{d-2}\frac{d^{d-2}}{dt^{d-2}}\psi(t)}$ is decreasing and convex on $(0,\infty)$.
	\par
	According to \textcite{mcneilneslehova2009}, an Archimedean copula $C$ admits a density $c$ if and only if $\psis{}{(d-1)}$ exists and is absolutely continuous on $(0,\infty)$. In this case, $c$ is given by
	\begin{align}
		c(\bm{u})=\psis{}{(d)}(t(\bm{u}))\prod_{j=1}^d(\psii)^\prime(u_j),\ \bm{u}\in(0,1)^d,\label{c}
	\end{align}
	where $t(\bm{u})=\sum_{j=1}^d\psi(u_j)$.
	\par
	We mainly assume $\psi$ to be \textit{completely monotone}, meaning that $\psi$ is continuous on $[0,\infty]$ and $\smash[t]{(-1)^k\frac{d^k}{dt^k}\psi(t)\ge 0}$ for all $k\in\IN_{0}$, $t\in(0,\infty)$, so that $\psi$ is the Laplace-Stieltjes transform of a distribution function $F$ on the positive real line, that is, $\psi=\LS[F]$; see Bernstein's Theorem in \textcite[p.\ 439]{feller1971}. The class of all such generators is denoted by $\Psi_\infty$ and it is clear that a $\psi\in\Psi_\infty$ generates an Archimedean copula in any dimensions $d$ and that its density exists.
	\par
	There are several well-known parametric generator families; see \textcite[pp.\ 116]{nelsen2007}, also referred to as \textit{Archimedean families}. Among the most widely used in applications are those of Ali-Mikhail-Haq (``A''), Clayton (``C''), Frank (``F''), Gumbel (``G''), and Joe (``J''); see Table~\ref{tab.gen}. We consider these families as working examples throughout this work. Detailed information about the corresponding distribution functions $F$ is given in \textcite{hofert2011b} and references therein. Note that these one-parameter families can be extended to allow for more parameters, for example, via outer power transformations. Furthermore, there are Archimedean families which are naturally given by more than a single parameter. Examples for both cases are given in Section \ref{sec.multi.param}.
	\begin{table}[htbp] 
		\centering
		\begin{tabularx}{\textwidth}{@{\extracolsep{\fill}}c@{\extracolsep{0mm}}c@{\extracolsep{1mm}}c@{\extracolsep{-3mm}}c}
			\toprule
			\multicolumn{1}{c}{Family}&\multicolumn{1}{c}{Parameter}&\multicolumn{1}{c}{$\psi(t)$}&\multicolumn{1}{c}{$V\sim F=\LSi[\psi]$}\\
			\midrule
				A&$\theta\in[0,1)$&$(1-\theta)/(\exp(t)-\theta)$&$\Geo(1-\theta)$\\
				C&$\theta\in(0,\infty)$&$(1+t)^{-{1/\theta}}$&$\Gamma(1/\theta,1)$\\
				F&$\theta\in(0,\infty)$&$-\log\bigl(1-(1-e^{-\theta})\exp(-t)\bigr)/\theta$&$\Log(1-e^{-\theta})$\\
				G&$\theta\in[1,\infty)$&$\exp(-t^{1/\theta})$&$\S(1/\theta,1,\cos^\theta(\pi/(2\theta)),\I_{\{\theta=1\}};1)$\\
				J&$\theta\in[1,\infty)$&$1-(1-\exp(-t))^{1/\theta}$&$\Sibuya(1/\theta)$\\
			\bottomrule 
		\end{tabularx}
		\setcapwidth{\textwidth}%
		\caption{Well-known one-parameter Archimedean generators $\psi$ with corresponding distributions $F=\LSi[\psi]$.}
		\label{tab.gen}
	\end{table}
	\par
	Table~\ref{tab.prop} summarizes properties concerning Kendall's tau and the tail-dependence coefficients; see \textcite[p.\ 91]{joe1997}, \textcite{joehu1996}, and \textcite[p.\ 214]{nelsen2007} for the investigated Archimedean families. Here, $D_1(\theta)=\int_0^\theta t/(\exp(t)-1)\,dt/\theta$ denotes the \textit{Debye function of order one}. Note that these properties are often of interest in order to choose a suitable model which is then estimated. The construction of initial intervals in Section \ref{sec.ii} for the optimization of the likelihood is based on Kendall's tau.	
	\begin{table}[htbp]
	  \centering
	  \begin{tabularx}{\textwidth}{@{\extracolsep{\fill}}cccc}
	    \toprule
	    \multicolumn{1}{c}{Family}&\multicolumn{1}{c}{$\tau$}&\multicolumn{1}{c}{$\lambda_L$}&\multicolumn{1}{c}{$\lambda_U$}\\
	    \midrule
			A&$1-2(\theta+(1-\theta)^2\log(1-\theta))/(3\theta^2)$&0&0\\
			C&$\theta/(\theta+2)$&$2^{-1/\theta}$&0\\
			F&$1+4(D_1(\theta)-1)/\theta$&0&0\\
			G&$(\theta-1)/\theta$&0&$2-2^{1/\theta}$\\
			J&$1-4\sum_{k=1}^\infty 1/(k(\theta k+2)(\theta(k-1)+2))$&0&$2-2^{1/\theta}$\\
		 \bottomrule
	  \end{tabularx}
	  \caption{Kendall's tau and tail-dependence coefficients.}
	  \label{tab.prop}
	\end{table}		
\section{Maximum-likelihood estimation for Archimedean copulas}\label{sec.est}
\subsection{The pseudo maximum-likelihood estimator}
   Assume that we have given realizations $\bm{x}_i$, $i\in\{1,\dots,n\}$, of independent and identically distributed (``i.i.d.'') random vectors $\bm{X}_i$, $i\in\{1,\dots,n\}$, from a joint distribution function $H$ with Archimedean copula $C$ generated by $\psi$ and corresponding density $c$. The generator $\psi$ is assumed to belong to a parametric family $(\psis{\th}{})_{\th\in\Theta}$ with parameter vector $\th\in\Theta\subseteq\IR^p$, $p\in\IN$, and the true but unknown vector is $\th_0$ (similarly, $C=C_{\th_0}$ and $c=c_{\th_0}$). As usual, random vectors or random variables are denoted by upper-case letters, their realizations by lower-case letters.
	\par
	Before estimating $\th_0$, the first step is usually to estimate the marginal distribution functions. In a second step, one then estimates $\th_0$. This two-step approach is typically much easier to accomplish than estimating the parameters of the marginal distribution functions and the copula parameter vector simultaneously. Estimating the marginal distribution functions can be done either parametrically or non-parametrically. Based on maximum-likelihood estimation, the former approach is suggested by \textcite{joexu1996} and is known as \textit{inference functions for margins}. The latter approach is known as \textit{pseudo maximum-likelihood estimation} and is suggested by \textcite{genestghoudirivest1995}; see \textcite{kimsilvapullesilvapulle2007} for a comparison of maximum-likelihood estimation, the method of inference functions for margins, and pseudo maximum-likelihood estimation. 
	\par
	Following pseudo maximum-likelihood estimation, the marginal distribution functions are estimated by their empirical distribution functions $\hat{F}_{nj}(x)=\frac{1}{n}\sum_{k=1}^n\I_{\{x_{kj}\le x\}}$, $j\in\{1,\dots,n\}$, leading to the so-called \textit{pseudo-observations} $\hat{\bm{u}}_i=(\hat{u}_{i1},\dots,\hat{u}_{id})\T$, $i\in\{1,\dots,n\}$, where
	\begin{align}
		\hat{u}_{ij}=\frac{n}{n+1}\hat{F}_{nj}(x_{ij})=\frac{r_{ij}}{n+1},\ i\in\{1,\dots,n\},\ j\in\{1,\dots,d\}.\label{pseudo.obs}
	\end{align}
	Here, for each $j\in\{1,\dots,d\}$, $r_{ij}$ denotes the rank of $x_{ij}$ among all $x_{kj}$, $k\in\{1,\dots,n\}$. The asymptotically negligible scaling factor of $n/(n+1)$ is used to force the variates to fall inside the open unit hypercube to avoid problems with density evaluation at the boundaries of $[0,1]^d$. As usual, the pseudo-observations are interpreted as realizations of a random sample from $C$ (despite known issues of this interpretation such as the fact that the pseudo-observations are neither realizations of perfectly independent random vectors nor that the components are perfectly following a univariate standard uniform distribution) based on which the copula parameter vector $\th_0$ is estimated.
\subsection{Likelihood theory}\label{sec.lik}
   Maximum-likelihood estimation is based on the following theory. Given realizations $\bm{u}_i$, $i\in\{1,\dots,n\}$, of a random sample $\bm{U}_i$, $i\in\{1,\dots,n\}$, from the copula $C$ (in practice, $\bm{u}_i$ is taken as $\hat{\bm{u}}_i$, $i\in\{1,\dots,n\}$, in (\ref{pseudo.obs})), the \textit{likelihood} and \textit{log-likelihood} are defined by
   \begin{align*}
		L(\th;\bm{u}_1,\dots,\bm{u}_n)=\prod_{i=1}^nc_{\th}(\bm{u}_i)\quad\text{and}\quad l(\th;\bm{u}_1,\dots,\bm{u}_n)&=\sum_{i=1}^nl(\th;\bm{u}_i),
	\end{align*}
	respectively, where
	\begin{align*}
		l(\th;\bm{u}_i)=\log c_{\th}(\bm{u}_i)=\log\bigl((-1)^d\psis{\th}{(d)}(t_{\th}(\bm{u}))\bigr)+\sum_{j=1}^d\log(-(\psiis{\th})^\prime(u_{ij})).
	\end{align*}
	Here, the subscript $\th$ of $t(\bm{u})$ is used to stress the dependence of $t(\bm{u})$ on $\th$. The \textit{maximum-likelihood estimator} $\hat{\th}_n=\hat{\th}_n(\bm{u}_1,\dots,\bm{u}_n)$ can thus be found by solving the optimization problem
	\begin{align*}
		\hat{\th}_n=\argsup_{\th\in\Theta}l(\th;\bm{u}_1,\dots,\bm{u}_n).
	\end{align*}
	This optimization is typically done numerically. 
	\par
	Assuming the derivatives to exist, the \textit{score function} is defined as
	\begin{align*}
		s_{\th}(\bm{u})=\nabla l(\th;\bm{u})=\biggl(\frac{\partial}{\partial\theta_1}l(\th;\bm{u}),\dots,\frac{\partial}{\partial\theta_p}l(\th;\bm{u})\biggr)\T
	\end{align*}
	and the \textit{Fisher information} is 
	\begin{align*}%
		I(\th)=\IE_{\th}\bigl[s_{\th}(\bm{U})s_{\th}(\bm{U})\T\,\bigr]=\IE_{\th}\biggl[\biggl(\frac{\partial}{\partial\theta_i}l(\th;\bm{u})\frac{\partial}{\partial\theta_j}l(\th;\bm{u})\biggr)_{i,j\in\{1,\dots,p\}}\biggr]
	\end{align*}
	for $\bm{U}\sim C$. 
	\par
	Under regularity conditions (see \textcite[p.\ 281]{coxhinkley1974}, \textcite[pp.\ 384]{rohatgi1976}, \textcite[pp.\ 144]{serfling1980}, \textcite[p.\ 2146]{neweymcfadden1994}, \textcite[p.\ 421]{schervish1995}, \textcite[p.\ 449]{lehmanncasella1998}, \textcite[pp.\ 51]{vandervaart2000}, \textcite[p.\ 386]{bickeldoksum2000}, or \textcite[p.\ 118]{davison2003}), the following result holds.%
	\begin{theorem}\label{th.lik}
		\myskip
				\begin{enumerate}[label=(\arabic*),leftmargin=*,align=left,itemsep=0mm,topsep=1mm]
			\item (Strong) consistency of maximum-likelihood estimators:
			 		\begin{align*}
						\hat{\th}_n=\hat{\th}_n(\bm{U}_1,\dots,\bm{U}_n)\omu{P}{\longrightarrow}{\text{a.s.}}\th_0\ (n\to\infty).
					\end{align*}
			\item\label{th.norm} Asymptotic normality of maximum-likelihood estimators:
			 		\begin{align*}
						\sqrt{n}\,I(\th_0)^{1/2}(\hat{\th}_n-\th_0)\omu{d}{\longrightarrow}{}N(\bm{0},I_p),
					\end{align*}
					where $I_p$ denotes the identity matrix in $\IR^{p\times p}$.
		\end{enumerate}
	\end{theorem}
\subsection{Generator derivatives and copula density}
	Applying maximum-likelihood estimation requires an efficient strategy for evaluating the (log-)density of the parametric Archimedean copula family to be estimated. The most important part is to know how to access the generator derivatives. As mentioned in the introduction, this requires to know both a tractable algebraic form of the derivatives and a procedure to numerically evaluate the formulas in an efficient way in terms of precision and run time.
	\par
	As mentioned in the introduction, it is often stated that a computer algebra system can be used to access a generator's derivatives. Such an approach has typically two major flaws:
		\begin{enumerate}[label=(\arabic*),leftmargin=*,align=left,itemsep=0mm,topsep=1mm]
		\item It is not trivial and sometimes not possible for a computer algebra system to find derivatives of higher order;
		\item Even if formulas are obtained, they are usually not provided in a form which is both numerically stable and sufficiently fast to evaluate. 
	\end{enumerate}
	We experienced these flaws when we tried to access the 50th derivative of a Gumbel generator $\psis{\theta}{}(t)$ with parameter $\theta=1.25$ at $t=15$. On a MacBook Pro running Max OS X 10.6.6, we aborted Mathematica 8 after ten minutes without obtaining a result. Maple 14 lead to the values 10\,628, -29\,800, and others (without warning) when computing $\psis{1.25}{(50)}(15)$ several times. Note the chaotic behavior of this deterministic problem; the values should of course be equal and positive! MATLAB 7.11.0 did return the correct value of (roughly) 1057, but failed to access $\psis{1.25}{(100)}(15)$ (aborted after ten minutes). Let us stress that carelessly using such programs in simulations may lead to wrong results. Apart from numerical issues, the formulas for the derivatives obtained from computer algebra systems can become quite large and thus rather slow to evaluate. They are therefore not suitable in large-scale simulation studies, for example, for goodness-of-fit tests (or simulations of their performance) involving a parametric bootstrap.
	\par
   In the following theorem we derive explicit formulas for the generator derivatives for all Archimedean families given in Table~\ref{tab.gen}. 
	\begin{theorem}\label{gen.der}
			\myskip
				\begin{enumerate}[label=(\arabic*),leftmargin=*,align=left,itemsep=0mm,topsep=1mm]
			\item\label{gen.der.AMH} For the family of Ali-Mikhail-Haq, 
					\begin{align*}
						(-1)^d\psis{\theta}{(d)}(t)=\frac{1-\theta}{\theta}\sideset{}{_{-d}}\Li(\theta\exp(-t)),\ t\in(0,\infty),\ d\in\IN_0,
					\end{align*}
					where $\sideset{}{_{s}}\Li(z)$ denotes the \textit{polylogarithm of order $s$ at $z$}. 
			\item For the family of Clayton,
					\begin{align*}
						(-1)^d\psis{\theta}{(d)}(t)=(d-1+1/\theta)_d(1+t)^{-(d+1/\theta)},\ t\in(0,\infty),\ d\in\IN_0, %
					\end{align*}
					where $(d-1+1/\theta)_d=\prod_{k=0}^{d-1}(k+1/\theta)=\frac{\Gamma(d+1/\theta)}{\Gamma(1/\theta)}$ denotes the falling factorial.%
			\item For the family of Frank, 
					\begin{align*}
						(-1)^d\psis{\theta}{(d)}(t)=\frac{1}{\theta}\sideset{}{_{-(d-1)}}\Li((1-e^{-\theta})\exp(-t)),\ t\in(0,\infty),\ d\in\IN_0.
					\end{align*}
			\item\label{gen.der.G} For the family of Gumbel, 
					\begin{align*}
						(-1)^d\psis{\theta}{(d)}(t)=\frac{\psis{\theta}{}(t)}{t^d}\var{P}{G}{d,\theta}(t^\alpha),\ t\in(0,\infty),\ d\in\IN,
					\end{align*}
					where
					\begin{align*}
						\var{P}{G}{d,\theta}(x)&=\sum_{k=1}^d\var{a}{G}{dk}(\theta)x^k,\\
						\var{a}{G}{dk}(\theta)&=(-1)^{d-k}\sum_{j=k}^d\theta^{-j}s(d,j)S(j,k)=\frac{d!}{k!}\sum_{j=1}^k\binom{k}{j}\binom{\alpha j}{d}(-1)^{d-j},\ k\in\{1,\dots,d\},
					\end{align*}
					and $s$ and $S$ denote the \textit{Stirling numbers of the first kind} and the \textit{second kind}, respectively. 
			\item For the family of Joe, 
					\begin{align*}
						(-1)^d\psis{\theta}{(d)}(t)=\frac{\exp(-t)}{\theta(1-\exp(-t))^{1-1/\theta}}\var{P}{J}{d,\theta}\biggl(\frac{\exp(-t)}{1-\exp(-t)}\biggr),\ t\in(0,\infty),\ d\in\IN,
					\end{align*}
					where
					\begin{align*}
						\var{P}{J}{d,\theta}(x)&=\sum_{k=1}^d\var{a}{J}{dk}(\theta)x^{k-1},\\
						\var{a}{J}{dk}(\theta)&=S(d,k)(k-1-1/\theta)_{k-1}=S(d,k)\frac{\Gamma(k-\alpha)}{\Gamma(1-\alpha)},\ k\in\{1,\dots,d\}.
					\end{align*}
		\end{enumerate}
	\end{theorem}
	\begin{proof}
			\myskip
				\begin{enumerate}[label=(\arabic*),leftmargin=*,align=left,itemsep=0mm,topsep=1mm]
			\item The generator of the Archimedean family of Ali-Mikhail-Haq is of the form $\psis{\theta}{}(t)=\sum_{k=1}^\infty p_k\exp(-kt)$, $t\in[0,\infty)$, with probability mass function $(p_k)_{k=1}^\infty$ as given in Table~\ref{tab.gen}. This implies that $(-1)^d\psis{\theta}{(d)}(t)=\sum_{k=1}^\infty p_kk^d\exp(-kt)$ from which the statement easily follows from the definition of the polylogarithm as $\sideset{}{_{s}}\Li(z)=\sum_{k=1}^\infty z^k/k^s$.
			\item The result for Clayton is straightforward to obtain by taking the derivatives. 
			\item Similar to \ref{gen.der.AMH}. 
			\item Now consider Gumbel's family. Writing the generator in terms of the exponential series and differentiating the summands, leads to $\psis{\theta}{(d)}(t)=\sum_{k=1}^\infty(-1)^k/k!(\alpha k)_d t^{\alpha k-d}$, where $\alpha=1/\theta$. Since for $d\in\IN$, $(\alpha k)_d=\sum_{j=1}^ds(d,j)(\alpha k)^j$, one obtains $\psis{\theta}{(d)}(t)=t^{-d}\sum_{k=1}^\infty(-t^\alpha)^k/k!\sum_{j=1}^ds(d,j)(\alpha k)^j=t^{-d}\sum_{j=1}^d\alpha^js(d,j)\sum_{k=1}^\infty k^j(-t^\alpha)^k/k!$. Note that $\exp(-x)\sum_{k=0}^\infty k^jx^k/k!$ is the \textit{$j$th exponential polynomial} and equals $\sum_{k=0}^jS(j,k)$ $\cdot x^k$; see \textcite{boyadzhiev2009}. With $x=-t^\alpha$ and noting that the summand for $k=0$ is zero, we obtain $\psis{\theta}{(d)}(t)=\psis{\theta}{}(t)t^{-d}\sum_{j=1}^d\alpha^js(d,j)\sum_{k=1}^jS(j,k)(-t^\alpha)^k$. Interchanging the order of summation leads to $\psis{\theta}{(d)}(t)=\psis{\theta}{}(t)t^{-d}\sum_{k=1}^d(-t^\alpha)^k\sum_{j=k}^d\alpha^js(d,j)S(j,k)$ $=\psis{\theta}{}(t)\sum_{k=1}^dt^{\alpha k-d}(-1)^k\sum_{j=k}^d\alpha^js(d,j)S(j,k)$ from which the result about $(-1)^d\psis{\theta}{(d)}$ directly follows. For the last equality in the statement about $\var{a}{G}{dk}(\theta)$ note that $k!/d!\var{a}{G}{dk}(\theta)=(-1)^{d-k}k!/d!\sum_{j=0}^d\alpha^js(d,j)S(j,k)=(-1)^{d-k}/d!\sum_{j=0}^d\alpha^js(d,j)\sum_{l=0}^k$ $\cdot\binom{k}{l}(-1)^{k-l}l^j=(-1)^{d-k}/d!\sum_{l=0}^k\binom{k}{l}(-1)^{k-l}\sum_{j=0}^d(\alpha l)^js(d,j)=(-1)^d\sum_{l=0}^k\binom{k}{l}\binom{\alpha l}{d}$ $\cdot(-1)^l$ from which the result follows.
		   \item For Joe's family, $(-1)^d\psis{\theta}{(d)}(t)=(-1)^{d+1}\frac{d^d}{dt^d}(1-\exp(-t))^\alpha$, $d\in\IN$, where $\alpha=1/\theta$. Letting $x=\exp(-t)$, this equals $-(x\frac{d}{dx})^d(1-x)^\alpha$. The operator $x\frac{d}{dx}$ is investigated in \textcite{boyadzhiev2009}. It follows from the results there that $(-1)^d\psis{\theta}{(d)}(t)=-\sum_{k=1}^dS(d,k)(-x)^k(\alpha)_k(1-x)^{\alpha-k}=-(1-x)^\alpha\sum_{k=1}^dS(d,k)(\alpha)_k(-x/(1-x))^k$. Thus, $(-1)^d\psis{\theta}{(d)}(t)=\alpha(1-x)^\alpha\sum_{k=1}^dS(d,k)(k-1-\alpha)_{k-1}(x/(1-x))^k$. Resubstituting leads to the result as stated. 
		\end{enumerate}
	\end{proof}
	With the notation as in Theorem \ref{gen.der}, we obtain the following representations for the densities of the Archimedean families of Ali-Mikhail-Haq, Clayton, Frank, Gumbel, and Joe. 
	\begin{corollary}\label{c.formulas}
		\myskip
			\begin{enumerate}[label=(\arabic*),leftmargin=*,align=left,itemsep=0mm,topsep=1mm]
		\item For the family of Ali-Mikhail-Haq, 
				\begin{align*}
					c_{\theta}(\bm{u})=\frac{(1-\theta)^{d+1}}{\theta^2}\frac{\var{h}{A}{\theta}(\bm{u})}{\prod_{j=1}^du_j^2}\sideset{}{_{-d}}\Li(\var{h}{A}{\theta}(\bm{u})),
				\end{align*}
				where $\var{h}{A}{\theta}(\bm{u})=\theta\prod_{j=1}^d\frac{u_j}{1-\theta(1-u_j)}$.
		\item For the family of Clayton,
				\begin{align*}
					c_{\theta}(\bm{u})=\prod_{k=0}^{d-1}(\theta k+1)\biggl(\,\prod_{j=1}^du_j\biggr)^{-(1+\theta)}(1+t_\theta(\bm{u}))^{-(d+1/\theta)}.
				\end{align*}
		\item For the family of Frank, 
				\begin{align*}
					c_{\theta}(\bm{u})=\biggl(\frac{\theta}{1-e^{-\theta}}\biggr)^{d-1}\sideset{}{_{-(d-1)}}\Li(\var{h}{F}{\theta}(\bm{u}))\frac{\exp(-\theta\sum_{j=1}^du_j)}{\var{h}{F}{\theta}(\bm{u})},
				\end{align*}
				where $\var{h}{F}{\theta}(\bm{u})=(1-e^{-\theta})^{1-d}\prod_{j=1}^d(1-\exp(-\theta u_j))$.
		\item For the family of Gumbel, 
				\begin{align*}
					c_{\theta}(\bm{u})=\theta^dC_{\theta}(\bm{u})\frac{\prod_{j=1}^d(-\log u_j)^{\theta-1}}{t_\theta(\bm{u})^d\prod_{j=1}^du_j}\var{P}{G}{d,\theta}(t_\theta(\bm{u})^{1/\theta}).
				\end{align*} 
		\item For the family of Joe, 
				\begin{align*}
					c_{\theta}(\bm{u})=\theta^{d-1}\frac{\prod_{j=1}^d(1-u_j)^{\theta-1}}{(1-\var{h}{J}{\theta}(\bm{u}))^{1-1/\theta}}\var{P}{J}{d,\theta}\biggl(\frac{\var{h}{J}{\theta}(\bm{u})}{1-\var{h}{J}{\theta}(\bm{u})}\biggr),
				\end{align*}
				where $\var{h}{J}{\theta}(\bm{u})=\prod_{j=1}^d(1-(1-u_j)^\theta)$.
	\end{enumerate}
	\end{corollary}
	\begin{proof}
		The proof is tedious but straightforward to obtain from Formula (\ref{c}) and the results from Theorem \ref{gen.der}.
	\end{proof}
	The following remarks stress the importance of Theorem \ref{gen.der} and Corollary \ref{c.formulas}.
	\begin{remark}
			\myskip
			\begin{enumerate}[label=(\arabic*),leftmargin=*,align=left,itemsep=0mm,topsep=1mm]
			\item Recursive formulas for the generator derivatives for some Archimedean families were presented by \textcite{barbegenestghoudiremillard1996} and \textcite{wuvaldezsherris2007}. In contrast, Theorem \ref{gen.der} provides explicit formulas. As seen from Corollary \ref{c.formulas}, this allows us to explicitly compute the densities of the corresponding well-known and widely used Archimedean families, even in large dimensions. Furthermore, it allows us to compute conditional distribution functions based on these families and important statistical quantities such as the Kendall distribution function, which is of interest, for example, in goodness-of-fit testing; see \textcite{genestquessyremillard2006}, \textcite{genestremillardbeaudoin2009}, or \textcite{heringhofert2011}. Among others, note that extreme value copulas rarely have an explicit form of the density, the important Gumbel family can now be added to this list.
			\item The derivatives presented in Theorem \ref{gen.der} also play an important role in asymmetric extensions of Archimedean copulas. For example, consider a \textit{Khoudraji-transformed Archimedean copula} $C$, given by
					\begin{align*}
						C(\bm{u})=C_\psi(u_1^{\alpha_1},\dots,u_d^{\alpha_d})\Pi(u_1^{1-\alpha_1},\dots,u_d^{1-\alpha_d}),
					\end{align*}
					where $C_\psi$ denotes an Archimedean copula generated by $\psi$, $\Pi$ denotes the independence copula, and $\alpha_j\in[0,1]$, $j\in\{1,\dots,d\}$, are parameters. Given the generator derivatives, the density of a Khoudraji-transformed Archimedean copula is given by
					\begin{align*}
						c(\bm{u})=\ \sum_{\mathclap{J\subseteq\{1,\dots,d\}}}\ \psi_V^{(\lvert J\rvert)}\biggl(\,\sum_{j=1}^d\psi_V^{-1}(u_j^{\alpha_j})\biggr)\prod_{j\in J}\alpha_j(\psiis{V})^\prime(u_j^{\alpha_j})\prod_{j\notin J}(1-\alpha_j)u_j^{-\alpha_j}.
					\end{align*}
					This makes maximum likelihood estimation for these copulas feasible; see \textcite{hofertvrins2011} for an application.								
			\item As pointed out by \textcite[pp.\ 117]{hofert2010c}, new Archimedean copulas are often constructed with simple transformations of the generators addressed in Theorem \ref{gen.der}. The results in Theorem \ref{gen.der} might therefore carry over to other Archimedean families. In fact, one example for such a transformation is the outer power transformation addressed in Section \ref{sec.multi.param}. 
			\item For an Archimedean generator $\psi$ with unknown derivatives but known $F=\LSi[\psi]$, \textcite{hofertmaechlermcneil2011a} suggested to approximate $(-1)^d\psis{}{(d)}$ via
					\begin{align*}
						(-1)^d\psis{}{(d)}(t)\approx\frac{1}{m}\sum_{k=1}^mV_k^d\exp(-V_kt),\ t\in(0,\infty),
					\end{align*}
					where $V_k\sim F$, $k\in\{1,\dots,m\}$, are realizations of i.i.d.\ random variables following $F=\LSi[\psi]$. In the conducted simulation study, this approximation turned out to be quite accurate. Furthermore, it is typically straightforward to implement. However, such a Monte Carlo approach is of course slower than having a direct formula for the generator derivatives at hand. 
		\end{enumerate}
	\end{remark}
\section{Sample size $n$ vs dimension $d$}\label{sec.nvsd}
	The results of \textcite{hofertmaechlermcneil2011a} indicate that the root mean squared error (``RMSE'') is decreasing in the dimension for all other parameters (Archimedean family, dependence level measured by Kendall's tau, and sample size) fixed. This may be intuitive for exchangeable copulas since the curse of dimensionality is circumvented by symmetry. In this section we briefly investigate how the RMSE decreases in the dimension. Figure~\ref{fig.n.vs.d} shows a clear picture. For fixed Archimedean family (Ali-Mikhail-Haq (``AMH''), Clayton, Frank, Gumbel, and Joe), dependence level measured by Kendall's tau ($\tau\in\{0.25,0.5,0.75\}$), and sample size ($n\in\{20,50,100,200\}$), the RMSE (estimated based on $N=500$ replications) is decreasing in the dimension ($d\in\{5,10,20,50,100\}$). As the log-log plot further reveals, the decrease of the RMSE in the dimension $d$ is of the same order as in the sample size $n$, that is, the mean squared error (``MSE'') satisfies
	\begin{align*}
		\text{MSE}\propto\frac{1}{nd}.
	\end{align*}
   Although this behavior in the sample size $n$ is well-known, the behavior in the dimension $d$ is rather impressive since it contradicts the findings of \textcite{weiss2010}, for example. In the latter work, conclusions are drawn based on simulations only involving small dimensions. In small dimensions, however, numerical problems are often not (regarded) as severe as in larger dimensions. Sometimes, they are simply not solved correctly. However, according to our experience, we believe that the larger the dimension of interest is, the more involved numerical issues typically are. This will certainly become more important in the future as applications are often high-dimensional.
	\begin{figure}[htbp]
	 	\centering
	   \includegraphics[width=0.78\textwidth]{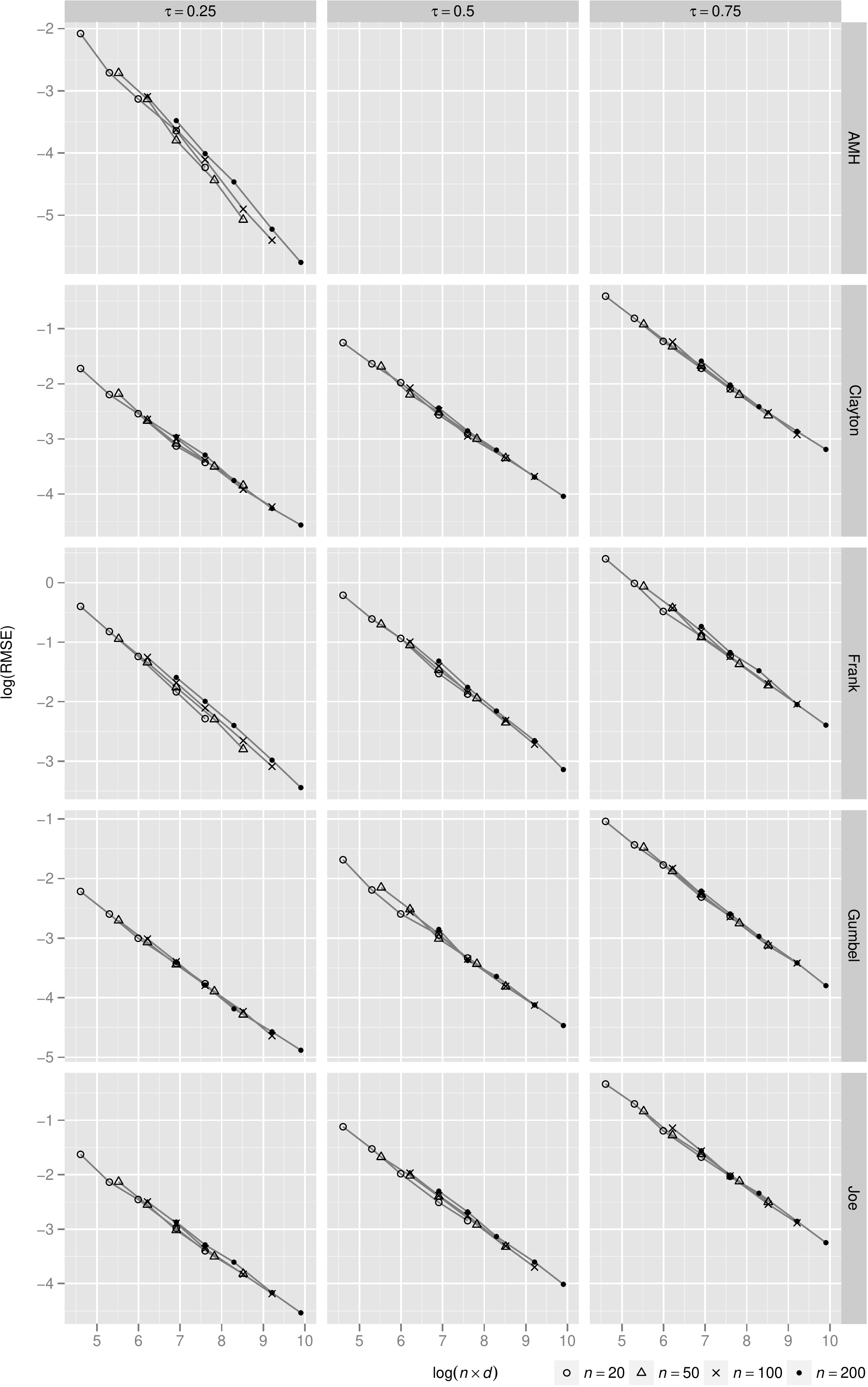}%
	 	\setcapwidth{0.78\textwidth}%
	 	\caption{log-RMSE ($N=500$ replications) as a function of the logarithm of $n\cdot d$. The plot indicates that the mean squared error satisfies $\text{MSE}\propto1/(nd)$ for all families and dependencies. Note that the family of AMH is limited to $\tau\in[0,1/3)$.} 
	 	\label{fig.n.vs.d}
	\end{figure}
\section{Constructing confidence intervals}\label{sec.ci}
   In this section, we describe different ways of how to obtain confidence intervals for the copula parameter vector $\th_0$. 
\subsection{Fisher information}
	It follows from Theorem \ref{th.lik} \ref{th.norm} that
	\begin{align*}
		(\hat{\th}_n-\th_0)\T\, nI(\th_0)(\hat{\th}_n-\th_0)\omu{d}{\longrightarrow}{}\chi^2_p\ (n\to\infty).
	\end{align*}
	This result remains valid if $I(\th_0)$ is replace by a consistent estimator $\widehat{I(\th_0)}$. Therefore, an asymptotic $1-\alpha$ \textit{confidence region} for $\th_0$ is given by
	\begin{align*}
		\Bigl\{\th\in\Theta\,:\,(\hat{\th}_n-\th)\T\, n\widehat{I(\th_0)}(\hat{\th}_n-\th)\le q_{\chi^2_p}(1-\alpha)\Bigr\},
	\end{align*}
	where $q_{\chi^2_p}(1-\alpha)$ denotes the $(1-\alpha)$-quantile of the chi-square distribution with $p$ degrees of freedom. In the one-parameter case, an asymptotic $1-\alpha$ \textit{confidence interval} for $\theta_0$ is given by
	\begin{align*}
		\Biggl[\hat{\theta}_n-\frac{z_{1-\alpha/2}}{\sqrt{n\widehat{I(\theta)}}},\hat{\theta}_n+\frac{z_{1-\alpha/2}}{\sqrt{n\widehat{I(\theta)}}}\Biggr],
	\end{align*}
	where $z_{1-\alpha/2}=\Phi^{-1}(1-\alpha/2)$ denotes the $(1-\alpha/2)$-quantile of the standard normal distribution function. 
	\par
	For the estimator $\widehat{I(\th_0)}$, there are several options, described in what follows. Assuming the derivatives to exist, the \textit{observed information} is defined as
	\begin{align*}
		J(\th;\bm{u}_1,\dots,\bm{u}_n)&=-\nabla\nabla\T l(\th;\bm{u}_1,\dots,\bm{u}_n)=\sum_{i=1}^n-\nabla\nabla\T l(\th;\bm{u}_i)\omu{}{=}{p=1}\sum_{i=1}^n-\frac{d^2}{d\theta^2}l(\theta;\bm{u}_i).
	\end{align*}
	Under regularity conditions (see the references in Section \ref{sec.lik}), the Fisher information satisfies
	\begin{align*}
		I(\th)=\IE[J(\th;\bm{U})]=\IE[-\nabla\nabla\T l(\th;\bm{U})]\omu{}{=}{p=1}\IE\biggl[-\frac{d^2}{d\theta^2}l(\theta;\bm{U})\biggr],
	\end{align*}
	that is, the Fisher information is the negative Hessian of the score function. From this and the definition of the Fisher information, the following choices for $\widehat{I(\th_0)}$ naturally arise (see also \textcite[pp.\ 2157]{neweymcfadden1994} including conditions for consistency):
	\begin{align}
		I(\hat{\th}_n)&=\IE_{\hat{\th}_n}\bigl[s_{\hat{\th}_n}(\bm{U})s_{\hat{\th}_n}(\bm{U})\T\,\bigr]\label{I.theta.hat}\\		\hat{I}^{(1)}(\hat{\th}_n)&=\frac{1}{n}\sum_{i=1}^ns_{\hat{\th}_n}(\bm{u}_i)s_{\hat{\th}_n}(\bm{u}_i)\T\label{I.1.hat.theta.hat}\\
		\hat{I}^{(2)}(\hat{\th}_n)&=\frac{1}{n}\sum_{i=1}^nJ(\hat{\th}_n;\bm{u}_i)=\frac{1}{n}\sum_{i=1}^n-\nabla\nabla\T l(\hat{\th}_n;\bm{u}_i)\label{I.2.hat.theta.hat}
	\end{align}
	The \textit{expected information} $I(\hat{\th}_n)$	is often difficult to obtain. Furthermore, \textcite{efronhinkley1978} argue for $\hat{I}^{(2)}(\hat{\th}_n)$ in favor of $I(\hat{\th}_n)$. The estimator $\hat{I}^{(1)}(\hat{\th}_n)$ is found much less in the literature, a reference being \textcite[p.\ 2157]{neweymcfadden1994}. The reason why we state it here is that there are cases where the second-order partial derivatives are (much) more complicated to access than the first-order ones based on the score function. 
		\par
		The following proposition provides the score functions for the one-parameter Archimedean families given in Table~\ref{tab.gen}.
		\begin{proposition}\label{score}
			\myskip
				\begin{enumerate}[label=(\arabic*),leftmargin=*,align=left,itemsep=0mm,topsep=1mm]
			\item For the family of Ali-Mikhail-Haq, 
					\begin{align*}
						s_{\th}(\bm{u})=-\frac{d+1}{1-\theta}-\frac{1}{\theta}+\var{b}{A}{\theta}(\bm{u})+\biggl(\var{b}{A}{\theta}(\bm{u})+\frac{1}{\theta}\biggr)\frac{\sideset{}{_{-(d+1)}}\Li(\var{h}{A}{\theta}(\bm{u}))}{\sideset{}{_{-d}}\Li(\var{h}{A}{\theta}(\bm{u}))},
					\end{align*}
					where $\var{b}{A}{\theta}(\bm{u})=\sum_{j=1}^d\frac{1-u_j}{1-\theta(1-u_j)}$.
			\item For the family of Clayton,
					\begin{align*}
						s_{\th}(\bm{u})=\sum_{k=0}^{d-1}\frac{k}{\theta k+1}-\sum_{j=1}^d\log u_j+\frac{1}{\theta^2}\log(1+t_\theta(\bm{u}))-(d+1/\theta)\frac{t_\theta(\bm{u})}{1+t_\theta(\bm{u})}.
					\end{align*}
			\item For the family of Frank, 
					\begin{align*}
						s_{\th}(\bm{u})&=\frac{d-1}{\theta}-\sum_{j=1}^d\frac{u_j}{1-\exp(-\theta u_j)}+\biggl(\sum_{j=1}^d\frac{u_j\exp(-\theta u_j)}{1-\exp(-\theta u_j)}-\frac{(d-1)e^{-\theta}}{1-e^{-\theta}}\biggr)\\
						 &\phantom{={}}\cdot\frac{\sideset{}{_{-d}}\Li(\var{h}{F}{\theta}(\bm{u}))}{\sideset{}{_{-(d-1)}}\Li(\var{h}{F}{\theta}(\bm{u}))}.
					\end{align*}
			\item For the family of Gumbel, 
					\begin{align*}
							s_{\th}(\bm{u})&=\frac{d-\log C_\theta(\bm{u})\log(-\log C_\theta(\bm{u}))}{\theta}-\var{b}{G}{\theta}(\bm{u})\biggl(d-\frac{\log C_\theta(\bm{u})}{\theta}\biggr)\\
							&\phantom{={}}+\sum_{j=1}^d\log(-\log u_j)+\frac{\var{Q}{G}{d,\theta,\bm{u}}(t_\theta(\bm{u})^{1/\theta})}{\theta\var{P}{G}{d,\theta}(t_\theta(\bm{u})^{1/\theta})},
					\end{align*}		
					where $\var{b}{G}{\theta}(\bm{u})=\sum_{j=1}^d\log(-\log u_j)\psii(u_j)/t_\theta(\bm{u})$ and $\var{Q}{G}{d,\theta,\bm{u}}(x)=\sum_{k=1}^d\var{a}{G}{dk}(\theta,\bm{u})x^k$ with $\var{a}{G}{dk}(\theta,\bm{u})=k\bigl(\var{b}{G}{\theta}(\bm{u})-\frac{1}{\theta}\log t_\theta(\bm{u})\bigr)\var{a}{G}{dk}(\theta)-(-1)^{d-k}\sum_{j=k}^djs(d,j)S(j,k)\theta^{-j}$.
			\item For the family of Joe, 
					\begin{align*}
							s_{\th}(\bm{u})&=\frac{d-1}{\theta}+\sum_{j=1}^d\log(1-u_j)-\frac{\log(1-\var{h}{J}{\theta}(\bm{u}))}{\theta^2}+\frac{(1-\frac{1}{\theta})\var{h}{J}{\theta}(\bm{u})}{1-\var{h}{J}{\theta}(\bm{u})}\var{b}{J}{\theta}(\bm{u})\\
						&\phantom{={}}+\frac{\var{Q}{J}{d,\theta,\bm{u}}\bigl(\var{h}{J}{\theta}(\bm{u})/(1-\var{h}{J}{\theta}(\bm{u}))\bigr)}{\var{P}{J}{d,\theta}\bigl(\var{h}{J}{\theta}(\bm{u})/(1-\var{h}{J}{\theta}(\bm{u}))\bigr)},									
					\end{align*}
					where $\var{b}{J}{\theta}(\bm{u})=\sum_{j=1}^d\frac{-\log(1-u_j)(1-u_j)^\theta}{1-(1-u_j)^\theta}$ and $\var{Q}{J}{d,\theta,\bm{u}}(x)=\sum_{k=1}^d\var{a}{J}{dk}(\theta,\bm{u})x^{k-1}$ with $\var{a}{J}{dk}(\theta,\bm{u})=\var{a}{J}{dk}(\theta)\bigl(\frac{1}{\theta}\sum_{j=1}^{k-1}\frac{1}{\theta j-1}+(k-1)\var{b}{J}{\theta}(\bm{u})/(1-\var{h}{J}{\theta}(\bm{u}))\bigr)$.
		\end{enumerate}
		\end{proposition}
		\begin{proof}
			The proof is quite tedious but straightforward to obtain from Corollary \ref{c.formulas}.
		\end{proof}
\subsection{Likelihood-based confidence intervals}
   Confidence regions or confidence intervals can also be constructed solely based on the likelihood function (without requiring its derivatives). For this, the \textit{likelihood ratio statistic} is used, defined as
	\begin{align*}
		W(\th;\bm{u}_1,\dots,\bm{u}_n)=2(l(\hat{\th}_n;\bm{u}_1,\dots,\bm{u}_n)-l(\th;\bm{u}_1,\dots,\bm{u}_n)),
	\end{align*}
	As \textcite[p.\ 126]{davison2003} notes, the likelihood ratio statistic asymptotically follows a chi-square distribution, meaning that
	\begin{align*}
		W(\th_0;\bm{U}_1,\dots,\bm{U}_n)\omu{d}{\longrightarrow}{}\chi^2_p\ (n\to\infty).
	\end{align*}
	Based on this result, an asymptotic $1-\alpha$ \textit{confidence region} for $\th_0$ is given by
	\begin{align}
		\bigl\{\th\in\Theta\,:\,l(\th;\bm{u}_1,\dots,\bm{u}_n)\ge l(\hat{\th}_n;\bm{u}_1,\dots,\bm{u}_n)-q_{\chi^2_p}(1-\alpha)/2\bigr\}.\label{W}
	\end{align}	
	\par
	If only a sub-vector $\th_0^\text{i}\in\Theta_{p^\text{i}}\subseteq\IR^{p^\text{i}}$ of components of $\th_0=({\th_0^\text{i}}\T\,,{\th_0^\text{n}}\T)\T$ are of interest ($\th_0^\text{i}$ and $\th_0^\text{n}$ are referred to as \textit{parameters of interest} and \textit{nuisance parameters}, respectively), an asymptotic confidence region for $\th_0^i$ follows from a similar argument to before, based on the \textit{profile log-likelihood}
	\begin{align*}
		l_{p_i}(\th^\text{i};\bm{u}_1,\dots,\bm{u}_n)=\sup_{\th^\text{n}}l\biggl(\vec{\th^\text{i}}{\th^\text{n}};\bm{u}_1,\dots,\bm{u}_n\biggr)=l\biggl(\vec{\th^\text{i}}{\hat{\th}_n^{\text{n},\th^\text{i}}};\bm{u}_1,\dots,\bm{u}_n\biggr),
	\end{align*}
	where $\hat{\th}_n^{\text{n},\th^\text{i}}$ is the maximum-likelihood estimator of $\th_0^\text{n}$ given $\th^\text{i}$. Under regularity conditions, the \textit{generalized likelihood ratio statistic} 
	\begin{align*}
		W_{p^\text{i}}(\th^\text{i};\bm{u}_1,\dots,\bm{u}_n)=2\biggl(l(\hat{\th}_n;\bm{u}_1,\dots,\bm{u}_n)-l\biggl(\vec{\th^\text{i}}{\hat{\th}_n^{\text{n},\th^\text{i}}};\bm{u}_1,\dots,\bm{u}_n\biggr)\biggr)
	\end{align*}
	satisfies	
	\begin{align*}
		W_{p^\text{i}}(\th_0^\text{i};\bm{U}_1,\dots,\bm{U}_n)\omu{d}{\longrightarrow}{}\chi^2_{p^\text{i}}\ (n\to\infty).
	\end{align*}	
	An asymptotic $1-\alpha$ \textit{confidence region} for $\th_0^\text{i}$ is thus given by
	\begin{align*}
		\bigl\{\th^\text{i}\in\Theta_{p^\text{i}}\,:\,l_{p^\text{i}}(\th^\text{i};\bm{u}_1,\dots,\bm{u}_n)\ge l_{p^\text{i}}(\hat{\th}_n^\text{i};\bm{u}_1,\dots,\bm{u}_n)-q_{\chi^2_{p^\text{i}}}(1-\alpha)/2\bigr\},
	\end{align*}	
	where
	\begin{align*}
		\hat{\th}_n^\text{i}=\argsup_{\th^\text{i}\in\Theta^\text{i}}l_{p^\text{i}}(\th^\text{i};\bm{u}_1,\dots,\bm{u}_n).
	\end{align*}
	This will be used in Section \ref{sec.multi.param} to construct confidence intervals for multi-parameter families. 
	\par
	\begin{example}\label{ex}
		The left-hand side of Figure \ref{fig.C} shows the log-likelihood of a Clayton copula based on a 100-dimensional sample of size $n=100$ with parameter $\theta_0=2$ such that the corresponding bivariate population version of Kendall's tau equals $\tau(\theta_0)=0.5$. The maximum-likelihood estimator is denoted by $\hat{\theta}_n$ and the lower and upper endpoints of the likelihood-based 0.95 confidence interval by $\theta_l^{0.95}$ and $\theta_u^{0.95}$, respectively. The right-hand side of Figure \ref{fig.C} shows the profile likelihood plot for the same sample. Similarly for Figure \ref{fig.G} which shows the log-likelihood and profile likelihood plot for the 100-dimensional Gumbel family with parameter $\theta_0=2$ such that Kendall's tau equals $\tau(\theta_0)=0.5$. 
		\begin{figure}[htbp]
		 	\centering
		   \includegraphics[width=0.477\textwidth]{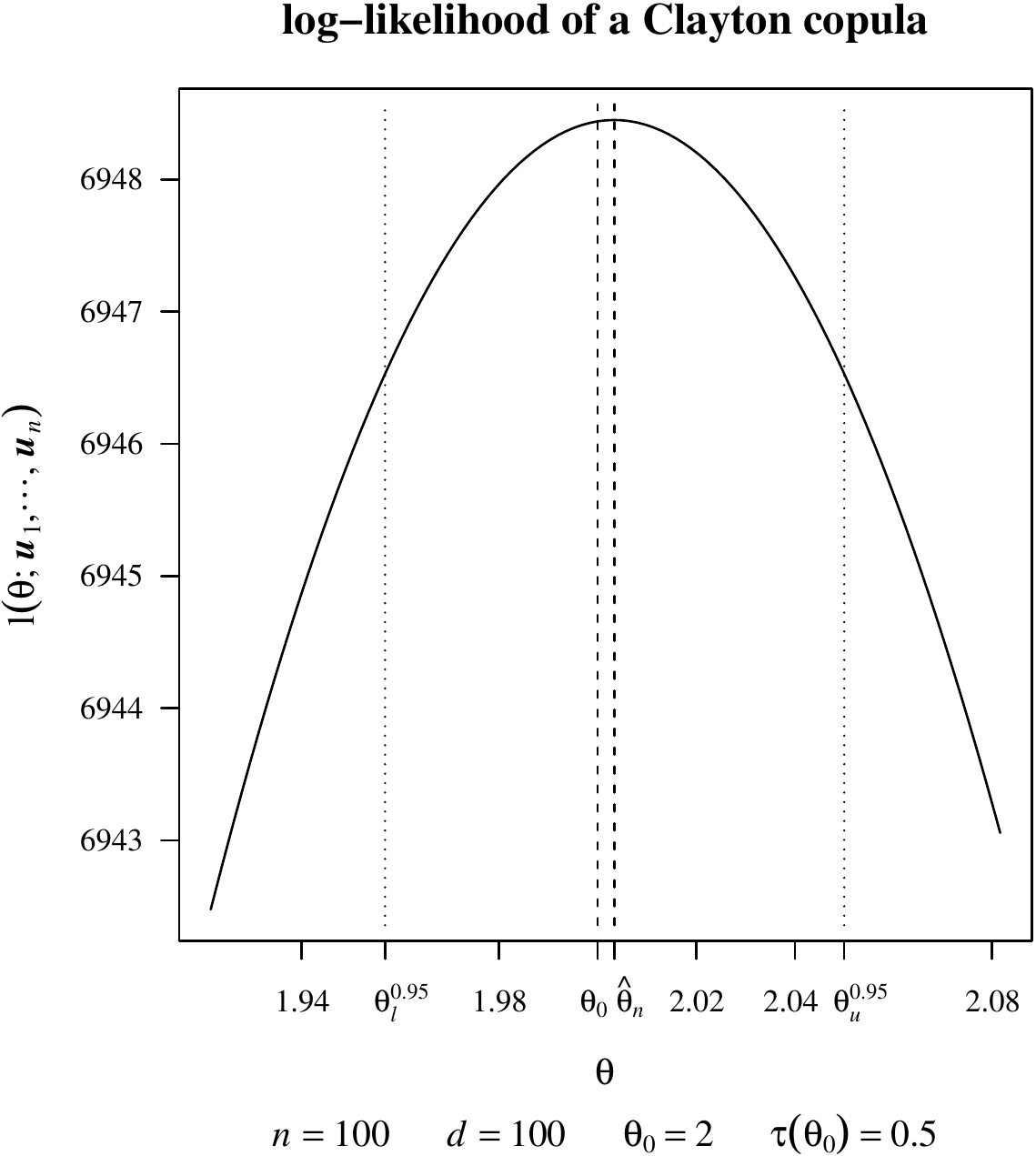}%
		   \hfill
		 	\includegraphics[width=0.50\textwidth]{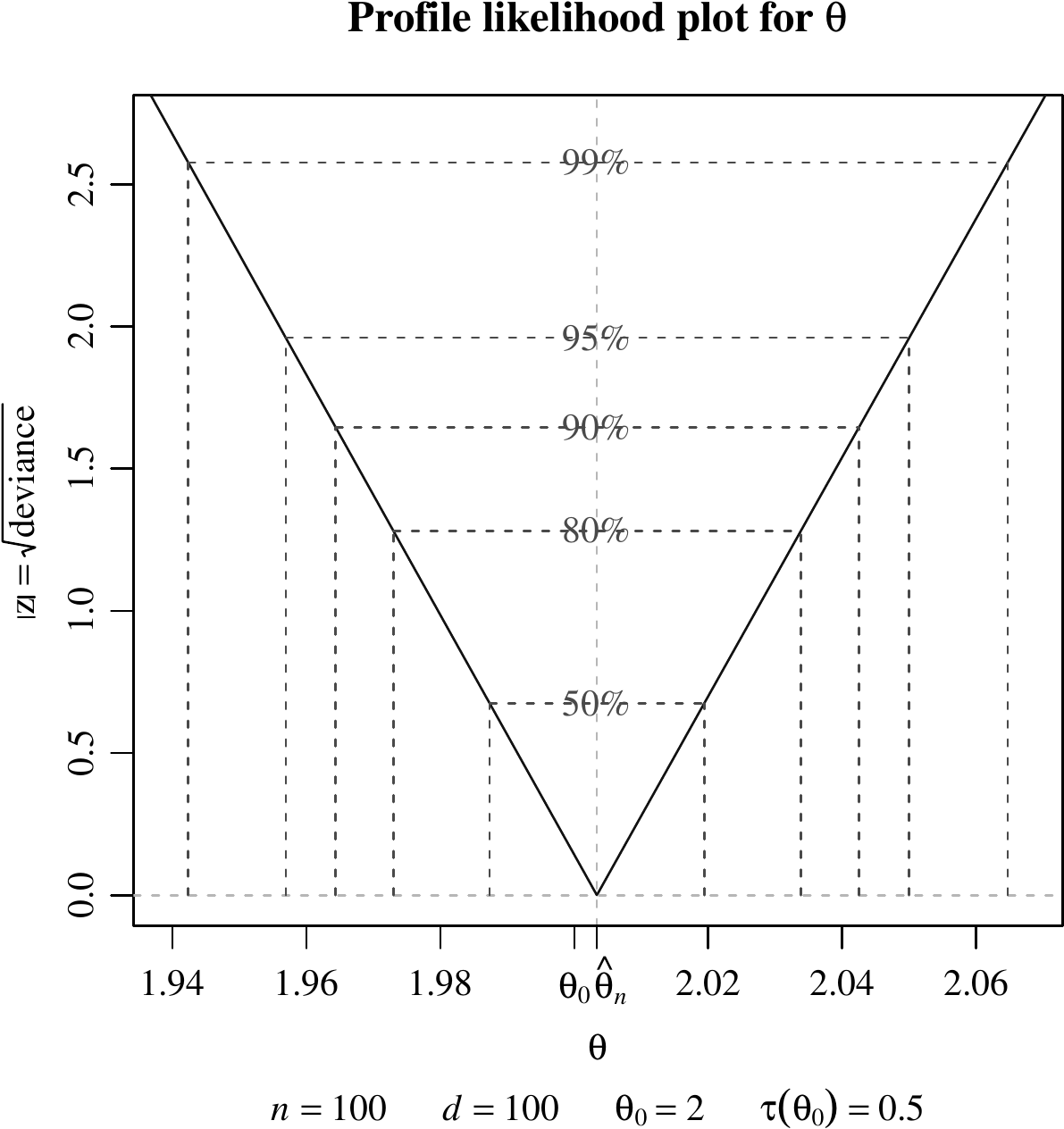}%
		 	\setcapwidth{\textwidth}%
		 	\caption{Plot of the log-likelihood of a Clayton copula (left) based on a sample of size $n=100$ in dimension $d=100$ with parameter $\theta_0=2$ such that Kendall's tau equals $0.5$. Corresponding profile likelihood plot (right).} 
		 	\label{fig.C}
		\end{figure}
		\begin{figure}[htbp]
		 	\centering
		   \includegraphics[width=0.477\textwidth]{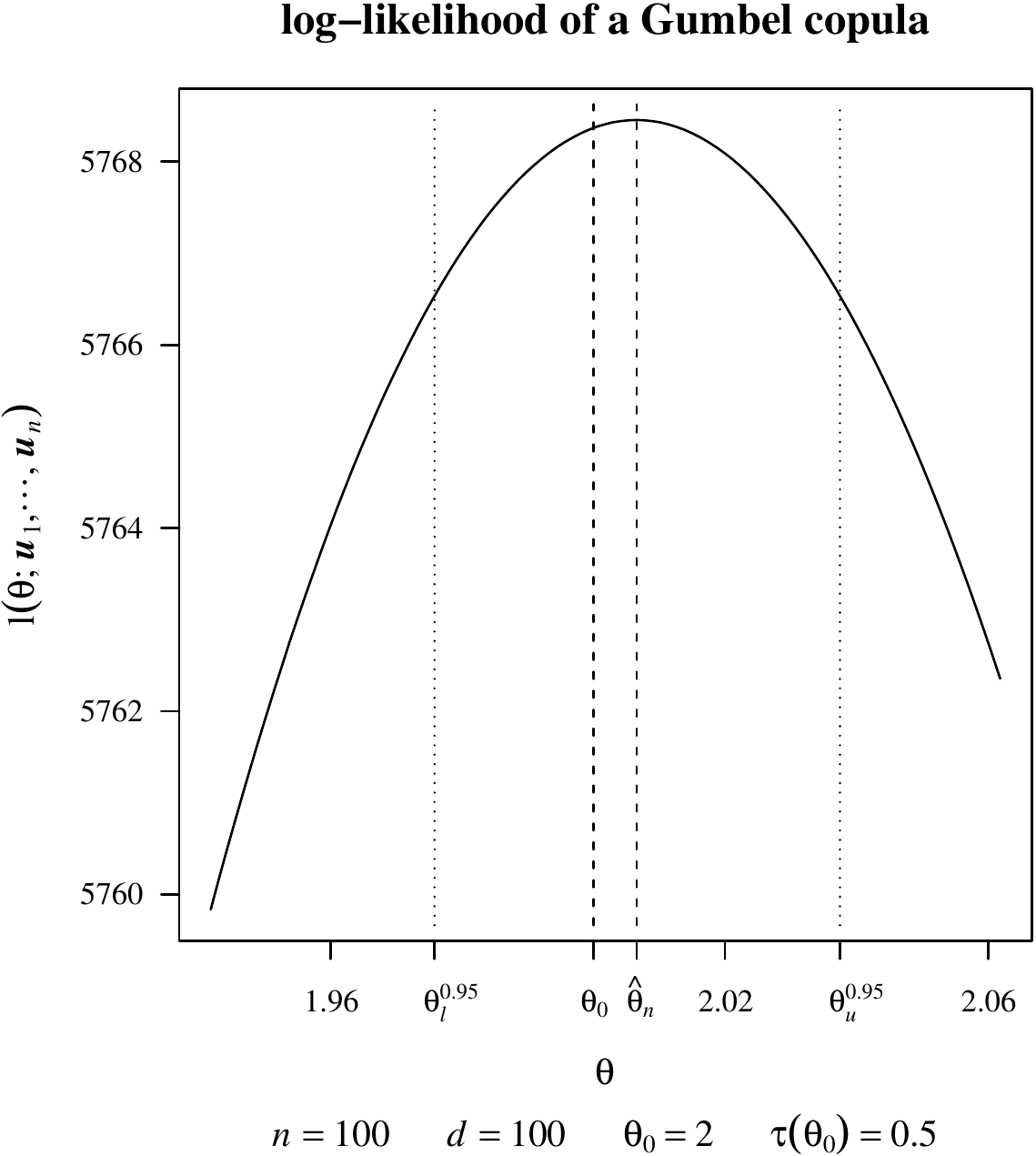}%
		   \hfill
		 	\includegraphics[width=0.50\textwidth]{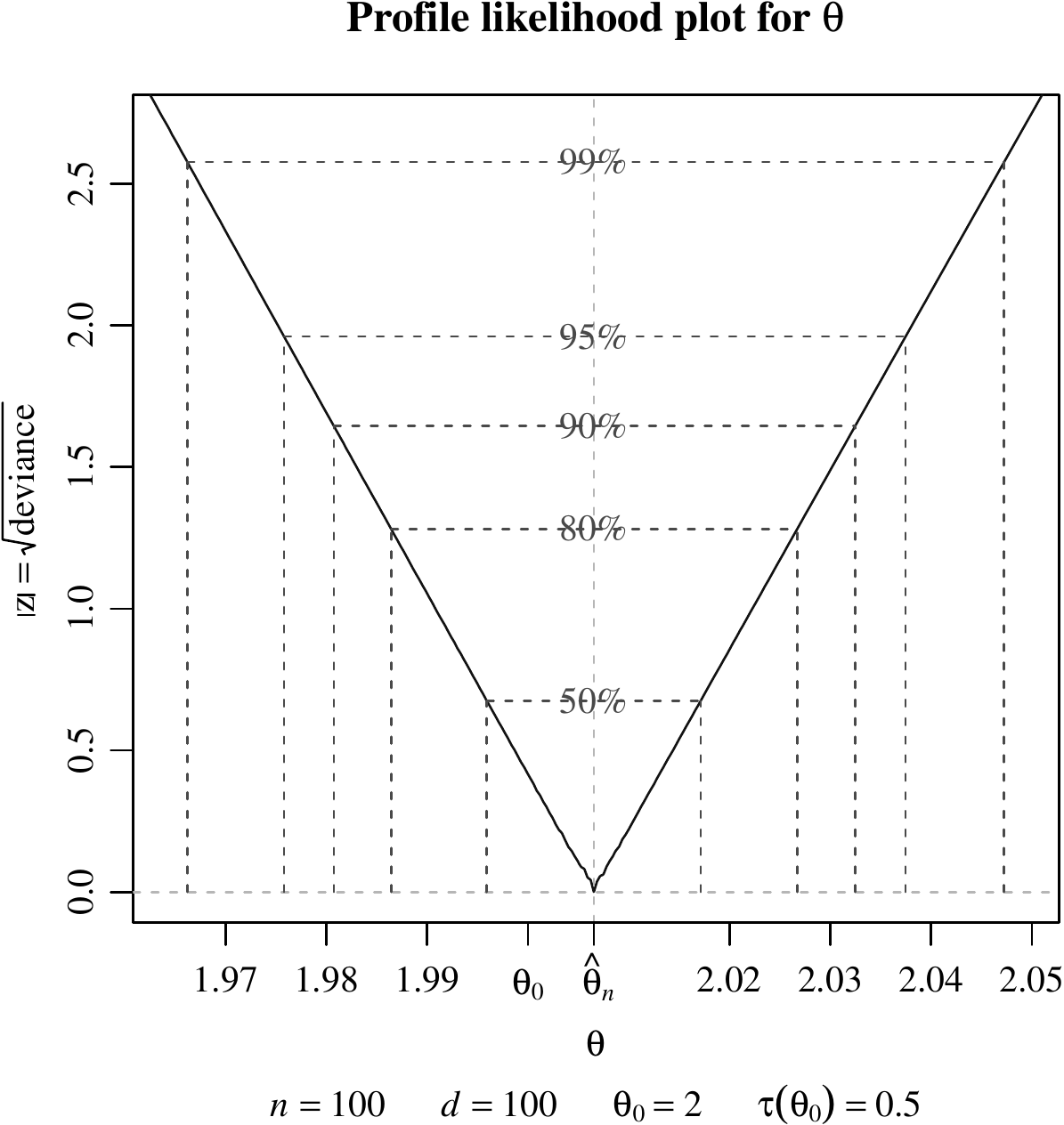}%
		 	\setcapwidth{\textwidth}%
		 	\caption{Plot of the log-likelihood of a Gumbel copula (left) based on a sample of size $n=100$ in dimension $d=100$ with parameter $\theta_0=2$ such that Kendall's tau equals $0.5$. Corresponding profile likelihood plot (right).} 
		 	\label{fig.G}
		\end{figure}
   \end{example}
\subsection{A simulation study to access the coverage probability}
	In this section, we compare the different approaches for obtaining (asymptotic) confidence regions and intervals. For this, we conduct a simulation study to access the coverage probability. The methods for obtaining confidence intervals based on the Fisher information are denoted by ``$I(\hat{\theta}_n)$'' for (\ref{I.theta.hat}), ``$\hat{I}^{(1)}(\hat{\theta}_n)$'' for (\ref{I.1.hat.theta.hat}), and ``$\hat{I}^{(2)}(\hat{\theta}_n)$'' for (\ref{I.2.hat.theta.hat}); the likelihood-based approach (\ref{W}) by ``$W$''.
	\par
	As can be seen from Proposition \ref{score}, already the score functions can be quite complicated. In order to be able to investigate the method $\hat{I}^{(2)}(\hat{\theta}_n)$ based on the observed information, we only consider the Clayton family, for which
	\begin{align*}
		\nabla\nabla\T l(\th;\bm{u})&=-\sum_{k=0}^{d-1}\biggl(\frac{k}{\theta k+1}\biggr)^2+\frac{2}{\theta^2}\biggl(\frac{t^\prime_\theta(\bm{u})}{1+t_\theta(\bm{u})}-\frac{1}{\theta}\log(1+t_\theta(\bm{u}))\biggr)\\
		 &\phantom{={}}+(d+1/\theta)\Biggl(\biggl(\frac{t^\prime_\theta(\bm{u})}{1+t_\theta(\bm{u})}\biggr)^2-\frac{\sum_{j=1}^d(\log u_j)^2u_j^{-\theta}}{1+t_\theta(\bm{u})}\Biggr),
	\end{align*}
	with $t^\prime_\theta(\bm{u})=\frac{d}{d\theta}t_\theta(\bm{u})=\sum_{j=1}^d(-\log u_j)u_j^{-\theta}$, that is, for which $\hat{I}^{(2)}(\hat{\theta}_n)$ can be easily computed. Our simulation study is based on the sample sizes $n\in\{100,400\}$ in the dimensions $d\in\{5,20\}$ for the dependencies $\tau\in\{0.25,0.5,0.75\}$. For each of these setups and each of the methods $I(\hat{\theta}_n)$, $\hat{I}^{(1)}(\hat{\theta}_n)$, $\hat{I}^{(2)}(\hat{\theta}_n)$, and $W$, we determine the proportion of cases among $N=1000$ replications for which the true parameter is contained in the computed confidence interval. Since the expected information is not known explicitly, we evaluate it by a Monte Carlo simulation based on samples of size $10\,000$.
	\par
	Table~\ref{tab.cov.prob} shows the results of the conducted simulation study. Overall, all methods work comparably well. Note that from a computational point of view, $\hat{I}^{(1)}(\hat{\theta}_n)$ is preferred to $I(\hat{\theta}_n)$ if the latter has to be evaluated based on a Monte Carlo simulation. Furthermore, $\hat{I}^{(2)}(\hat{\theta}_n)$ is typically difficult to evaluate, due to the complicated second order derivatives; the tractable Clayton family is certainly an exception. Even $\hat{I}^{(1)}(\hat{\theta}_n)$ may be (numerically) challenging for some families, as Proposition \ref{score} indicates. The likelihood based approach $W$ has several advantages. First, it is typically even simpler to evaluate than $\hat{I}^{(1)}(\hat{\theta}_n)$. Second, it may lead to asymmetric confidence intervals. Finally, by using a re-parameterization, it allows one to construct confidence intervals for quantities such as Kendall's tau or the tail-dependence coefficients (otherwise often obtained from the Delta Method based on the approximate normal distribution).
	\begin{table}[htbp] 
    \centering\footnotesize
    \begin{tabularx}{\textwidth}{@{\extracolsep{\fill}}d{1}{3}d{3}{0}d{1}{2}d{2}{0}d{2}{1}d{2}{1}d{2}{1}d{2}{1}}
      \toprule
      \multicolumn{4}{l}{Coverage probabilities for Clayton (in \%)}&\multicolumn{4}{c}{Method for obtaining confidence intervals}\\
		\cmidrule(lr{0.4em}){5-8}	
		\multicolumn{1}{c}{$1-\alpha$}&\multicolumn{1}{c}{$n$}&\multicolumn{1}{c}{$\tau$}&\multicolumn{1}{c}{$d$}&\multicolumn{1}{c}{$I(\hat{\theta}_n)$}&\multicolumn{1}{c}{$\hat{I}^{(1)}(\hat{\theta}_n)$}&\multicolumn{1}{c}{$\hat{I}^{(2)}(\hat{\theta}_n)$}&\multicolumn{1}{c}{$W$}\\ 
      \midrule
	   0.95  & 100 & 0.25 & 5  & 95.6 & 95.4 & 96.0 & 95.6 \\ 
	        &     &      & 20 & 94.6 & 95.4 & 94.9 & 94.9 \\ 
	   \addlinespace[2pt]      &     & 0.5  & 5  & 94.1 & 94.2 & 94.0 & 94.0 \\ 
	        &     &      & 20 & 95.9 & 96.7 & 95.8 & 95.8 \\ 
	   \addlinespace[2pt]      &     & 0.75 & 5  & 95.7 & 95.9 & 95.6 & 95.7 \\ 
	        &     &      & 20 & 95.8 & 95.9 & 95.9 & 95.9 \\ 
	   \addlinespace[4pt]      & 400 & 0.25 & 5  & 94.8 & 94.9 & 95.1 & 95.1 \\ 
	        &     &      & 20 & 95.7 & 96.3 & 96.0 & 96.0 \\ 
	   \addlinespace[2pt]      &     & 0.5  & 5  & 95.1 & 95.2 & 95.2 & 95.0 \\ 
	        &     &      & 20 & 95.6 & 95.3 & 95.3 & 95.2 \\ 
	   \addlinespace[2pt]      &     & 0.75 & 5  & 94.8 & 94.5 & 95.0 & 94.9 \\ 
	        &     &      & 20 & 94.7 & 95.0 & 94.7 & 94.7 \\ 
	   \addlinespace[6pt]0.99  & 100 & 0.25 & 5  & 98.8 & 98.7 & 99.1 & 99.1 \\ 
	        &     &      & 20 & 98.9 & 98.4 & 98.9 & 98.9 \\ 
	   \addlinespace[2pt]      &     & 0.5  & 5  & 98.2 & 98.7 & 98.4 & 98.4 \\ 
	        &     &      & 20 & 99.4 & 99.2 & 99.3 & 99.3 \\ 
	   \addlinespace[2pt]      &     & 0.75 & 5  & 99.1 & 99.1 & 99.2 & 99.2 \\ 
	        &     &      & 20 & 98.8 & 98.7 & 98.8 & 98.8 \\ 
	   \addlinespace[4pt]      & 400 & 0.25 & 5  & 98.7 & 98.7 & 98.8 & 98.8 \\ 
	        &     &      & 20 & 99.3 & 99.1 & 99.3 & 99.3 \\ 
	   \addlinespace[2pt]      &     & 0.5  & 5  & 99.0 & 98.9 & 99.0 & 99.0 \\ 
	        &     &      & 20 & 99.6 & 99.7 & 99.6 & 99.6 \\ 
	   \addlinespace[2pt]      &     & 0.75 & 5  & 98.7 & 98.6 & 98.7 & 98.7 \\ 
	        &     &      & 20 & 99.1 & 99.3 & 99.1 & 99.1 \\ 
	   \addlinespace[6pt]0.995 & 100 & 0.25 & 5  & 99.5 & 99.5 & 99.7 & 99.5 \\ 
	        &     &      & 20 & 99.4 & 99.4 & 99.5 & 99.5 \\ 
	   \addlinespace[2pt]      &     & 0.5  & 5  & 99.2 & 99.2 & 99.4 & 99.3 \\ 
	        &     &      & 20 & 99.8 & 99.8 & 99.9 & 99.9 \\ 
	   \addlinespace[2pt]      &     & 0.75 & 5  & 99.5 & 99.5 & 99.5 & 99.5 \\ 
	        &     &      & 20 & 99.6 & 99.1 & 99.6 & 99.5 \\ 
	   \addlinespace[4pt]      & 400 & 0.25 & 5  & 99.0 & 99.1 & 99.3 & 99.3 \\ 
	        &     &      & 20 & 99.6 & 99.4 & 99.6 & 99.6 \\ 
	   \addlinespace[2pt]      &     & 0.5  & 5  & 99.1 & 99.3 & 99.2 & 99.2 \\ 
	        &     &      & 20 & 99.9 & 99.8 & 99.9 & 99.9 \\ 
	   \addlinespace[2pt]      &     & 0.75 & 5  & 99.3 & 99.3 & 99.3 & 99.2 \\ 
	        &     &      & 20 & 99.7 & 99.6 & 99.8 & 99.7 \\ 
      \bottomrule
    \end{tabularx}
    \setcapwidth{\textwidth}%
    \caption{Simulated coverage probabilities for Clayton's family based on $N=1000$ replications.}
    \label{tab.cov.prob}
  \end{table}
\section{Multi-parameter families}\label{sec.multi.param}
	The one-parameter generators of Ali-Mikhail-Haq, Clayton, Frank, Gumbel, and Joe can easily be extended to allow for more parameters, for example, by so-called outer power transformations or even more general generator transformations; see \textcite{hofert2010b}, \textcite{hofert2010c}, or \textcite{hofert2011a}. In this section, we investigate an outer power Clayton copula and the Archimedean GIG family and apply maximum-likelihood estimation for estimating the copula parameters. Both of these families are available via the \textsf{R} package \texttt{nacopula} so that the interested reader can easily follow our calculations. The computations carried out in this section were run on a Mac mini under Mac OS X Version 10.6.6 with a 2.66\,GHz Intel Core 2 Duo processor and 4\,GB 1067\,MHz DDR3 memory. The \textsf{R} version used is 2.12.1.
\subsection{Finding initial intervals}\label{sec.ii}
	Maximizing the log-likelihood $l$ is typically achieved by a numerical routine. These algorithms often require an initial interval (or an initial value, which can be derived from the former). This interval should be sufficiently large in order to contain the optimum, but also sufficiently small in order to find the optimum fast. Furthermore, one should be able to compute an initial interval in a small amount of time in comparison to the actual log-likelihood evaluations required for maximizing the log-likelihood. 
	\par		
	For Archimedean families with $\psis{\th}{}\in\Psi_\infty$, the measure of concordance Kendall's tau is a function in $\th$ which always maps to the unit interval; see, for example, \textcite[pp.\ 59]{hofert2010c}. It thus provides an intuitive ``distance'' in terms of concordance. For one-parameter families, one can thus typically choose an initial interval of the form
	\begin{align*}
		[\tau^{-1}(\max\{\hat{\tau}-h,\tau_l\}),\tau^{-1}(\min\{\hat{\tau}+h,\tau_u\})],
	\end{align*}
	where $h\in[0,1]$ is suitably chosen with intuitive interpretation as ``distance in concordance'' and $\tau_l$ and $\tau_u$ denote lower and upper admissible Kendall's tau for the families considered (in Example~\ref{ex} we used this technique to find an interval on which the log-likelihood is plotted; we took $\hat{\tau}$ as the correct value $\tau=0.5$, and used $h=0.01$ and $h=0.015$ for Clayton's and Gumbel's family, respectively). If the dimension is not too large, one can take the mean of pairwise sample versions of Kendall's tau as estimator $\hat{\tau}$ of Kendall's tau; see \textcite{berg2009}, \textcite{kojadinovicyan2010}, and \textcite{savutrede2010} for this estimator. Another option is a multivariate version of Kendall's tau; see \textcite[pp.\ 217]{jaworskidurantehaerdlerychlik2010}. A fast way, especially in large dimensions, is to utilize the explicit diagonal maximum-likelihood estimator 
	\begin{align*}
		\hat{\theta}_n^{\text{G}}=\frac{\log d}{\log n-\log\bigl(\sum_{i=1}^n-\log Y_i\bigr)},\ \text{where}\ Y_i=\max_{j\in\{1,\dots,d\}}U_{ij},\ i\in\{1,\dots,n\}.
	\end{align*}
	for Gumbel's family, see \textcite{hofertmaechlermcneil2011a}, and estimate Kendall's tau by $\tau^{\text{G}}(\hat{\theta}_n^{\text{G}})$, where $\tau^{\text{G}}(\theta)=(\theta-1)/\theta$ denotes Kendall's tau for Gumbel's family as a function in the parameter. Since the optimization for one-parameter families is typically not too time-consuming, one can also just maximize the log-likelihood on a reasonably large, fixed interval, for example $[\tau^{-1}(h_1),\tau^{-1}(h_2)]$, where $h_1$ and $h_2$ are suitably chosen constants in the range of $\tau$; see \textcite{hofertmaechlermcneil2011a}.
	\par
	For multi-parameter Archimedean families, the log-likelihood is typically even more challenging to evaluate. An initial interval therefore also serves the purpose of reducing the parameter space to an area where the log-likelihood can be evaluate without numerical problems. The idea we present here to construct initial intervals for multi-parameter families is again based on Kendall's tau. In a first step, we estimate Kendall's tau by $\hat{\tau}_n$. To this end we apply the pairwise Kendall's tau estimator, which, due to the rather complicated log-likelihood evaluations does not take too much run time for the ten-dimensional examples considered below; another option would be to randomly select sub-columns of the data and apply the pairwise Kendall's tau estimator to this sub-data in order to reduce run time. Based on this estimator of Kendall's tau, we then construct an initial rectangle by three points. These points are determined via $\tau^{-1}(\hat{\tau}_n-h_-)$ and $\tau^{-1}(\hat{\tau}_n+h_+)$, that is, via certain positive numbers $h_-$ and $h_+$ (sufficiently small to ensure that $\hat{\tau}_n-h_-$ and $\hat{\tau}_n+h_+$ are in the range of admissible Kendall's tau). They allow for an intuitive interpretation as ``distance in (terms of) concordance'' and are independent of the parameterization of the family (since they measure distances in Kendall's tau and not in the underlying copula parameters). Now note that $\tau^{-1}$ is not uniquely defined for two- or more-parameter families. It is, however, if one fixes all but one parameter. By starting with one corner of the initial rectangle to be constructed and applying monotonicity properties of $\tau$ as a function in its parameters, one can thus construct an initial rectangle around the estimate $\hat{\tau}_n$ of $\tau(\th_0)$. More details are given in Sections \ref{sec.opC} and \ref{sec.GIG} for the two-parameter Archimedean families investigated. 	
\subsection{Outer power copulas}\label{sec.opC}
	If $\psi\in\Psi_\infty$, so is $\tilde{\psi}(t)=\psi(t^{1/\beta})$ for all $\beta\in[1,\infty)$, since the composition of a completely monotone function with a non-negative function that has a completely monotone derivative is again completely monotone; see \textcite[p.\ 441]{feller1971}. The copula family generated by $\tilde{\psi}$ is referred to as \textit{outer power family}.  
	\par   
	The generator derivatives of $\tilde{\psi}(t)=\psi(t^{1/\beta})$ can be accessed with a formula about derivatives of compositions which dates back at least to \textcite{schloemilch1846}. According to this formula, 
	\begin{align*}
		(-1)^d\tilde{\psi}^{(d)}(t)=\var{P}{op}{}(t^{1/\beta})/t^d,\ d\in\IN,\quad\text{where}\quad\var{P}{op}{}(x)=\sum_{k=1}^d\var{a}{G}{dk}(\beta)(-1)^k\psi^{(k)}(x)x^k.
	\end{align*}
	Via (\ref{c}) and the form of $\var{a}{G}{dk}$ given in Theorem \ref{gen.der} \ref{gen.der.G} one can thus easily derive the density of an outer power copula.
	\par
	For sampling $\tilde{V}\sim\tilde{F}=\LSi[\tilde{\psi}]$, \textcite{hofert2011a} derived the stochastic representation
	\begin{align*}
		\tilde{V}=SV^\beta,\ S\sim\S(1/\beta,1,\cos^\beta(\pi/(2\beta)),\I_{\{\beta=1\}};1),\ V\sim F=\LSi[\psi].
	\end{align*}
	Note that $\tilde{V}$ can easily be sampled via the \textsf{R} package \texttt{nacopula} for all $\psi$ given in Table~\ref{tab.gen}.
	\par
	We consider the case where $\psi$ is Clayton's generator, so we obtain the two-parameter \textit{outer power Clayton copula} with generator $\tilde{\psi}(t)=(1+t^{1/\beta})^{-1/\theta}$. This copula, which generalizes the Clayton family, was successfully applied in \textcite{hofertscherer2011} in the context of pricing collateralized debt obligations. For this copula, Kendall's tau and the tail-dependence coefficients are given explicitly by
	\begin{align}
		\tau=\tau(\theta,\beta)=1-\frac{2}{\beta(\theta+2)},\quad\lambda_L=2^{-1/(\beta\theta)},\quad\lambda_U=2-2^{1/\beta}.\label{opC.prop}
	\end{align}
	Note the possibility to have upper tail dependence for this copula, which is not possible for a Clayton copula.
	\par
	The following algorithm describes a procedure for finding an initial interval for outer power Clayton copulas. The algorithm can easily be adapted to other outer power copulas, given that the base family (the family generated by $\psi$) is positively ordered in its parameter and admits a sufficiently large range of Kendall's tau.  
	\begin{algorithm}\label{alg.ii.increasing}
		\myskipalgo
		\linespread{1.22}\normalfont
		\begin{tabbing}
			\hspace{8mm} \= \hspace{4mm} \= \hspace{80mm} \= \kill
			(1) \> Choose $h_-,h_+\ge0$, and $\eps>0$.\\
			(2) \> Let the smallest $\beta$ be denoted by $\beta_l=1$.\\
			(3) \> Solve $\tau(\theta_u,\beta_l)=\min\{\hat{\tau}_n+h_+,1-\eps\}$ with respect to $\theta_u$.\\
			(4) \> Solve $\tau(\theta_l,\beta_l)=\max\{\hat{\tau}_n-h_-,\eps\}$ with respect to $\theta_l$.\\
			(5) \> Solve $\tau(\theta_l,\beta_u)=\min\{\hat{\tau}_n+h_+,1-\eps\}$ with respect to $\beta_u$.\\
			(6) \> Return the initial interval $I=[(\theta_l,\beta_l)\T,(\theta_u,\beta_u)\T]$.\\
		\end{tabbing}
	\end{algorithm}
	The idea behind Algorithm \ref{alg.ii.increasing} is to construct an initial rectangle by three points. First, the lower-right endpoint of the rectangle is constructed. Since $\tau(\theta,\beta)$ is an increasing function in both $\theta$ and $\beta$, the largest $\theta$ and the smallest $\beta$, that is, $(\theta_u,\beta_l)\T$, are chosen such that Kendall's tau equals $\hat{\tau}_n$ plus a small ``distance in concordance'' $h_+\ge0$ to ensure that $\theta_u$ is indeed an upper bound for $\theta$. The truncation done by $\eps>0$ is to obtain an admissible Kendall's tau range. Second, the lower-left endpoint is found. The monotonicity of $\tau$ justifies determining the minimal value $\theta_l$ for $\theta$ such that $\tau(\theta_l,\beta_l)=\max\{\hat{\tau}_n-h_-,\eps\}$, where $h_-\ge0$ is suitably chosen, similar to $h_+$. In the third and final step, the upper-left endpoint of the initial rectangle is determined. The maximal value $\beta_u$ for $\beta$ is determined in a similar fashion to the first step. Note that all equations can be solved explicitly due to the explicit form of Kendall's tau as given in (\ref{opC.prop}). 
	\par
	To access the performance of the maximum-likelihood estimator, we generate $N=1000$ times $n=100$ realizations of i.i.d.\ random vectors following $d$-dimensional outer power Clayton copulas. For demonstration purposes, we consider $d=10$. Furthermore, we consider three setups of dependencies: $\th=(\theta,\beta)\T=(1/3,8/7)\T$ resulting in a Kendall's tau of 0.25; $\th=(1,4/3)\T$ with corresponding Kendall's tau equal to 0.5; and $\th=(2,2)\T$ with Kendall's tau equal to 0.75. For finding initial intervals, Algorithm \ref{alg.ii.increasing} is applied with $\eps=0.005$, $h_-=0.4$, and $h_+=0$. The results are summarized in Table~\ref{tab.opC}, where ``RMSE'' denotes the root mean squared error as before and ``MUT'' denotes the mean user time (in seconds).  
	\begin{table}[htbp] 
   	\centering
		\begin{tabularx}{\textwidth}{@{\extracolsep{\fill}}cd{1}{2}d{1}{4}d{1}{4}d{1}{4}d{1}{4}d{2}{4}d{1}{4}d{2}{0}d{1}{2}}
      \toprule	
      &&&\multicolumn{2}{c}{$\hat{\theta}_n$}&&\multicolumn{2}{c}{$\hat{\beta}_n$}&\\
		\cmidrule(lr{0.4em}){4-5}\cmidrule(lr{0.4em}){7-8}
		\multicolumn{1}{c}{$n$}&\multicolumn{1}{c}{$\tau$}&\multicolumn{1}{c}{$\theta$}&\multicolumn{1}{c}{Bias}&\multicolumn{1}{c}{RMSE}&\multicolumn{1}{c}{$\beta$}&\multicolumn{1}{c}{Bias}&\multicolumn{1}{c}{RMSE}&\multicolumn{1}{c}{$\#$}&\multicolumn{1}{c}{MUT}\\ 
      \midrule
      100 & 0.25 & 0.3333 & 0.0073 & 0.0609 & 1.1429 & -0.0017 & 0.0429 & 42 & 0.7\,\text{s} \\ 
		100 & 0.5 & 1.0000 & 0.0082 & 0.1050 & 1.3333 & -0.0003 & 0.0613 & 39 & 0.6\,\text{s} \\
		100 & 0.75 & 2.0000 & 0.0107 & 0.1786 & 2.0000 & -0.0012 & 0.1088 & 47 & 0.7\,\text{s} \\ \addlinespace[4pt]
		500 & 0.25 & 0.3333 & 0.0025 & 0.0276 & 1.1429 & -0.0014 & 0.0189 & 42 & 1.2\,\text{s} \\
		500 & 0.5 & 1.0000 & 0.0026 & 0.0451 & 1.3333 & -0.0013 & 0.0268 & 38 & 1.1\,\text{s} \\
		500 & 0.75 & 2.0000 & 0.0031 & 0.0753 & 2.0000 & -0.0013 & 0.0483 & 49 & 1.3\,\text{s} \\
      \bottomrule
    \end{tabularx}
    \setcapwidth{\textwidth}%
    \caption{Summary statistics for estimating two-parameter outer power Clayton copulas.}
    \label{tab.opC}
  \end{table}
  \par
  Figure~\ref{fig.opC.wf.lp} shows a wire-frame plot (left) of the negative log-likelihood of a sample of size $n=100$ for the setup $\th=(1,4/3)\T$ ($\tau=0.5$) and the corresponding level plot (right). Both plots have the initial interval determined by Algorithm \ref{alg.ii.increasing} as domain and show both the true value $\th_0=(\theta_0,\beta_0)\T$ and the optimum $\hat{\th}_n=(\hat{\theta}_n,\hat{\beta}_n)\T$ as determined by the optimizer. 
  	\begin{figure}[htbp]
	 	\centering
	   \includegraphics[width=0.45\textwidth]{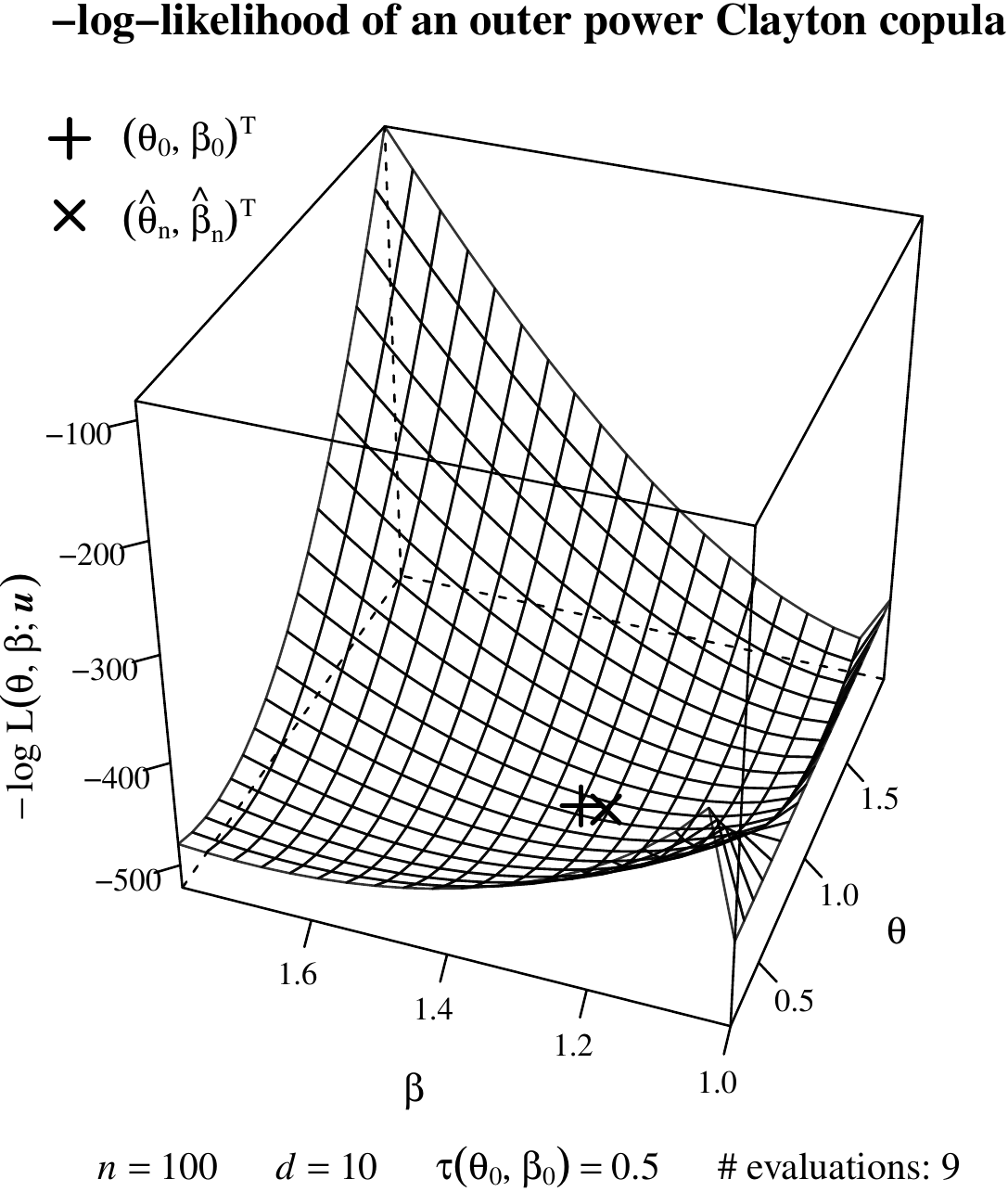}%
	   \hfill
	 	\includegraphics[width=0.537\textwidth]{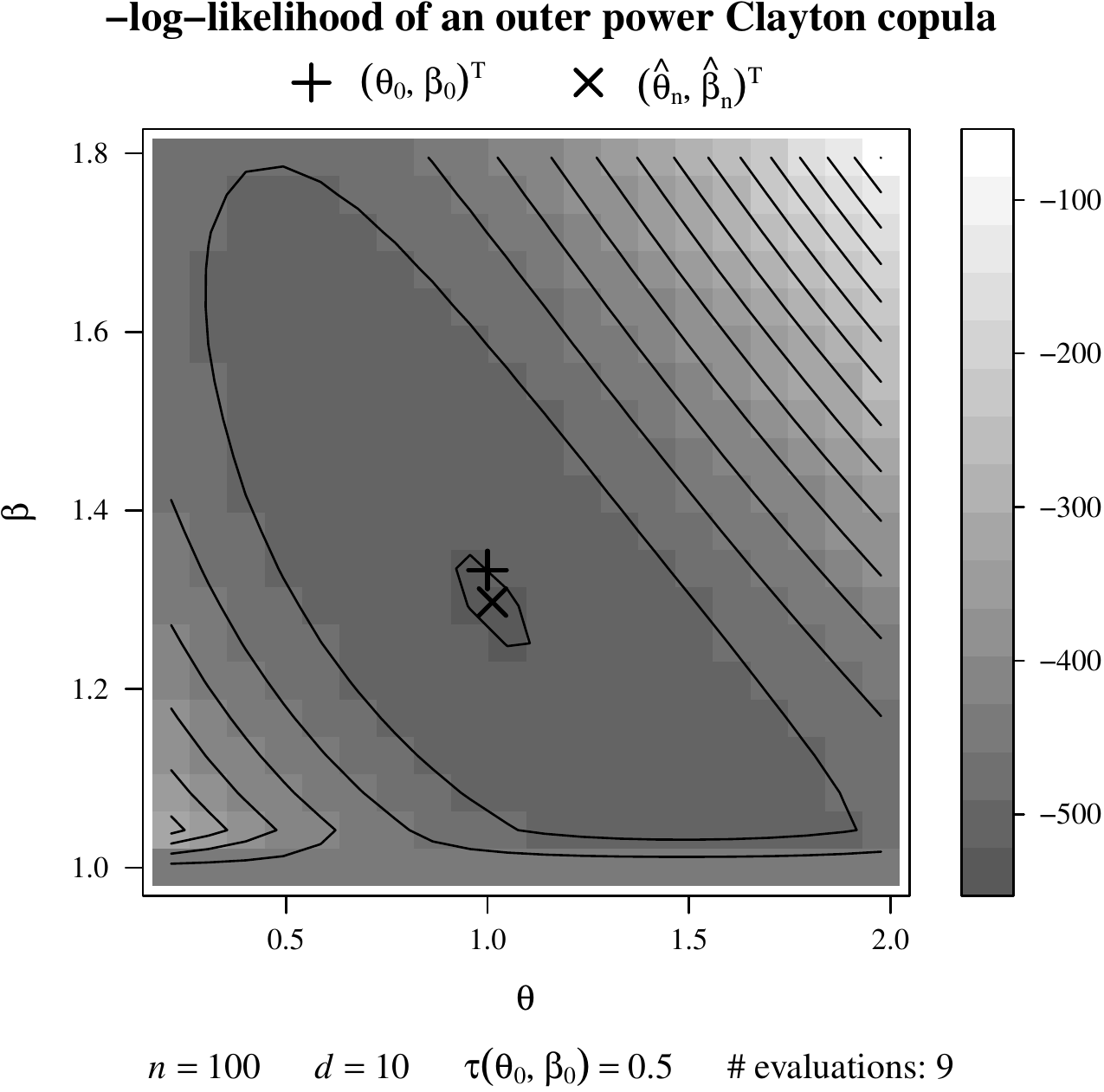}%
	 	\setcapwidth{\textwidth}%
	 	\caption{A wire-frame plot (left) and corresponding level plot (right) of the negative log-likelihood function for an outer power Clayton copula for a sample of size $n=100$ for $\th=(1,4/3)\T$ ($\tau=0.5$) with the computed initial interval as domain.} 
	 	\label{fig.opC.wf.lp}
	\end{figure}
	Figure \ref{fig.opC.profile} shows profile likelihood plots for the two parameters $\theta$ and $\beta$. 
	\begin{figure}[htbp]
	 	\centering
	   \includegraphics[width=0.48\textwidth]{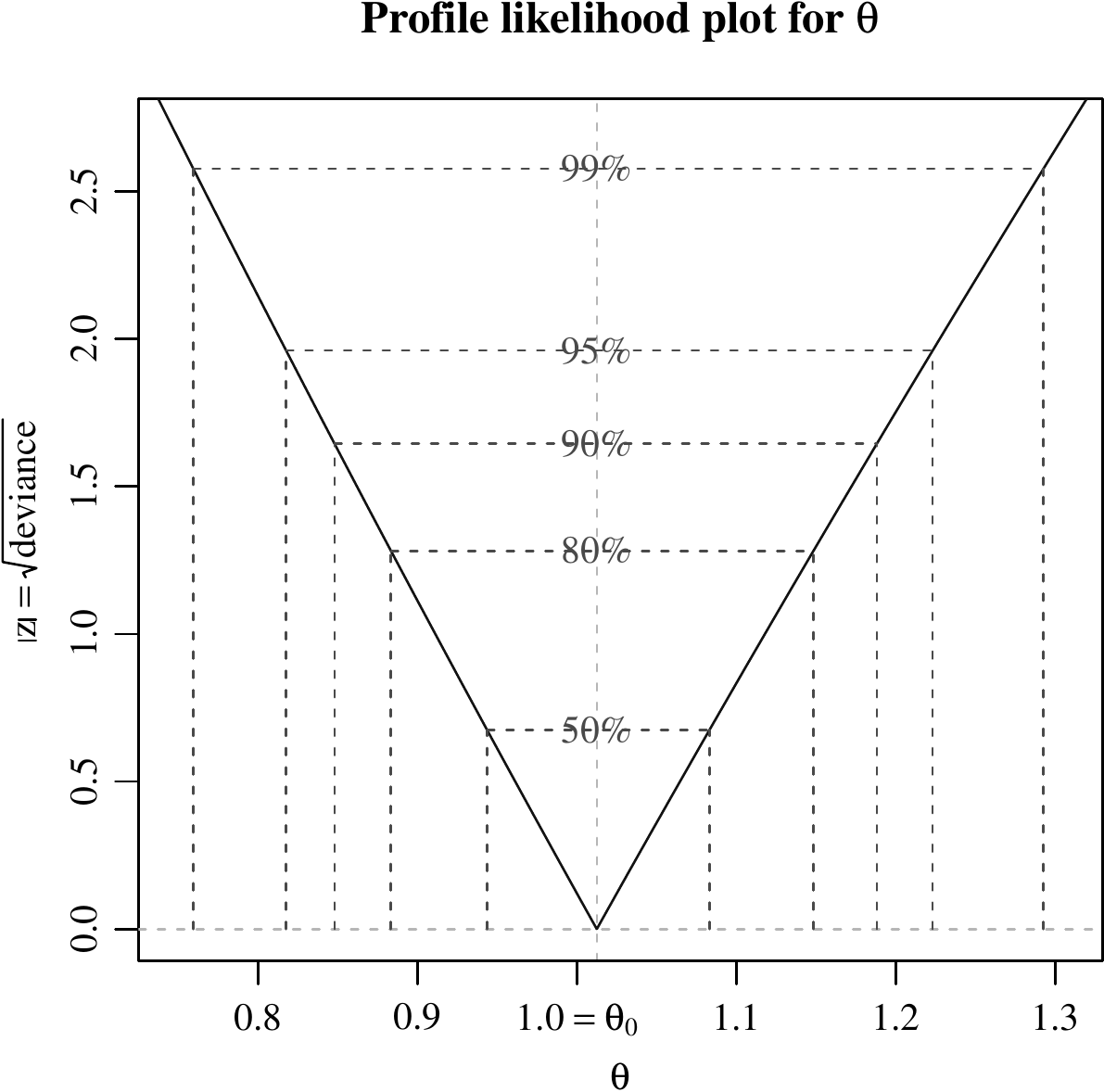}%
	   \hfill
	 	\includegraphics[width=0.48\textwidth]{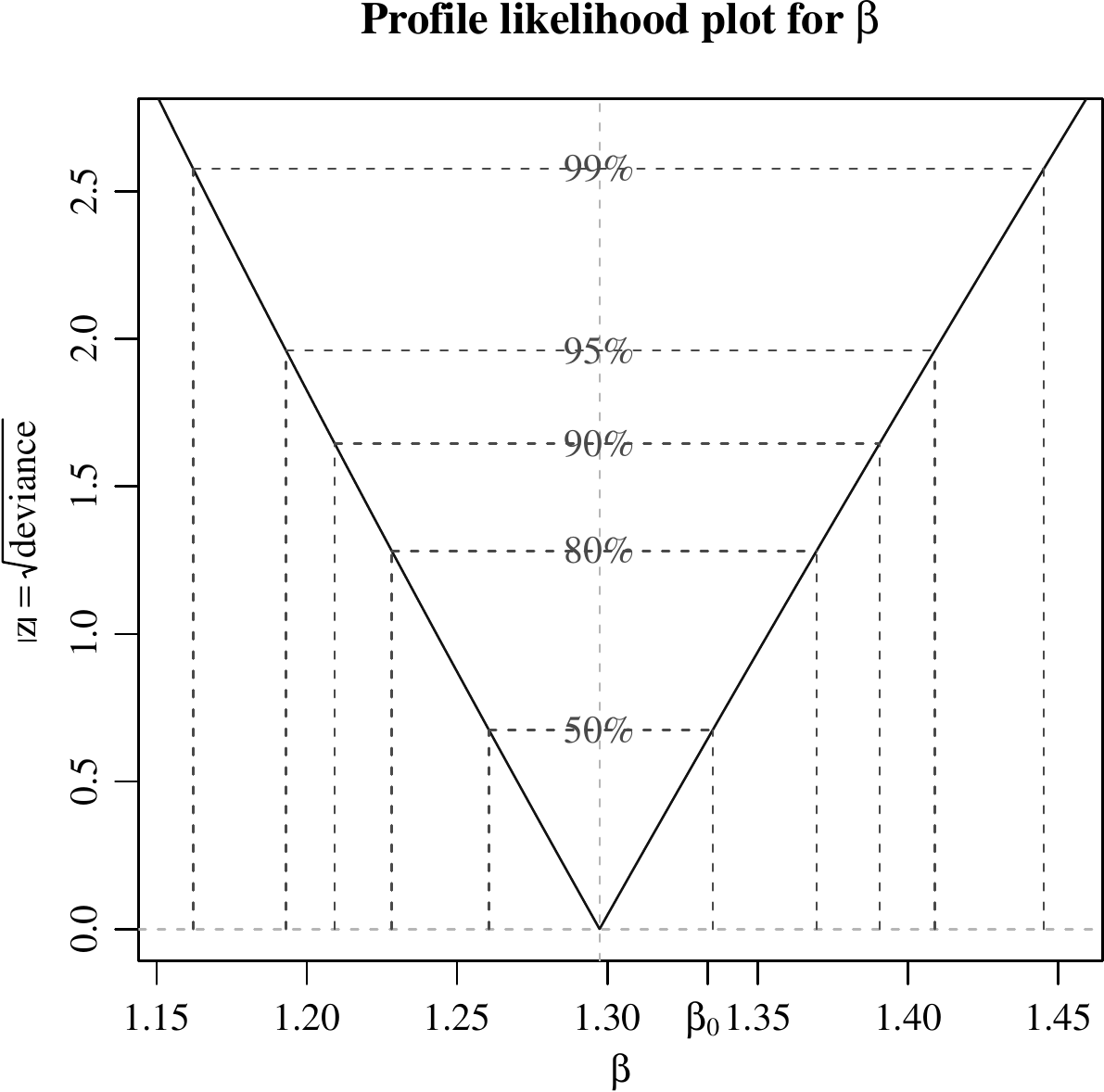}%
	 	\setcapwidth{\textwidth}%
	 	\caption{Profile likelihood plot for $\theta$ (left) and $\beta$ (right).} 
	 	\label{fig.opC.profile}
	\end{figure}
\subsection{The GIG family}\label{sec.GIG}
	An Archimedean family which naturally allows for two parameters can be constructed as follows. We start with the density of a \textit{generalized inverse Gaussian distribution} $\GIG(\nu,\phi,\chi)$, given by
	\begin{align*}
		g(x;\nu,\phi,\chi)=\frac{(\phi/\chi)^{\nu/2}}{2K_\nu(\sqrt{\phi\chi})}x^{\nu-1}\exp(-(\chi/x+\phi x)/2),\ x\in(0,\infty).
	\end{align*}
	Here, $\nu\in\IR$ with: $\phi\in[0,\infty)$, $\chi\in(0,\infty)$, if $\nu\in(-\infty,0)$; $\phi,\chi\in(0,\infty)$, if $\nu=0$; and $\phi\in(0,\infty)$, $\chi\in[0,\infty)$, if $\nu\in(0,\infty)$; see \textcite[p.\ 497]{mcneilfreyembrechts2005}. The function $K_\nu(t)=\int_0^\infty\cosh(\nu x)\exp(-t\cosh(x))\,dx$ denotes the \textit{modified Bessel function of the third kind} with parameter $\nu$. It is decreasing in $t$ and symmetric about zero in $\nu$. Furthermore, it is increasing in $\nu$ if $\nu\in(0,\infty)$. Another important property is
	\begin{align}
	   \lim_{t\searrow0}t^\nu K_\nu(t)=2^{\nu-1}\Gamma(\nu)\ \text{if}\ \nu\in(0,\infty).\label{h.lim}
	\end{align}
	\par
	Note that for a $\psi\in\Psi_\infty$, the generator $\psi(ct)$ generates the same Archimedean copula as $\psi(t)$ for all $c\in(0,\infty)$. Letting $c=\phi/2$ and $\theta=\sqrt{\phi\chi}$ leads to a comparably simple form of the generator of an \textit{Archimedean GIG copula} with parameter vector $\th=(\nu,\theta)\T$, given by
	\begin{align}
		\psi(t)=(1+t)^{-\nu/2} K_\nu(\theta\sqrt{1+t})/K_\nu(\theta),\ t\in[0,\infty),\ \nu\in\IR,\ \theta\in(0,\infty).\label{psi.GIG}
	\end{align}
	If we let 
	\begin{align*}
		h_{\nu_1,\nu_2,\theta}(t)=\frac{(\theta\sqrt{1+t})^{\nu_1}K_{\nu_1}(\theta\sqrt{1+t})}{\theta^{\nu_2}K_{\nu_2}(\theta)},\ \nu_1,\nu_2\in\IR,\ \theta\in(0,\infty),\ t\in[0,\infty),
	\end{align*}
	one obtains from (\ref{h.lim}) that $\lim_{\theta\searrow0}h_{\nu_1,\nu_2,\theta}(t)=2^{\nu_1-\nu_2}\Gamma(\nu_1)/\Gamma(\nu_2)$ for every $\nu_1,\nu_2\in(0,\infty)$ and $t\in[0,\infty)$. Since $\psi$ can be written as $\psi(t)=(1+t)^{-\nu}h_{\nu,\nu,\theta}(t)$, one obtains as limiting case a $\Gamma(\nu,1)$ distribution for $\theta\searrow0$ if $\nu\in(0,\infty)$, that is, a Clayton copula with parameter $1/\nu$.
	\par
	The density $f$ of $F=\LSi[\psi]$ is given by
	\begin{align*}
		f(x;\nu,\theta)=\frac{(2x)^{\nu-1}}{\theta^\nu K_\nu(\theta)}\exp(-(\theta^2/(2x)+2x)/2),\ x\in(0,\infty),
	\end{align*}
	so that $V=X/2\sim F$ with $X\sim\GIG(\nu,1,\theta^2)$. The GIG distribution can easily be sampled with the \textsf{R} package \texttt{Runuran}. For numerically computing $\psii$, note that $K_\nu(y)/K_\nu(x)<\exp(x-y)(y/x)^\nu$ for all $\nu\in(-1/2,\infty)$ and $0<x<y<\infty$ (see \textcite{paris1984}), so that $[0,(1-\log(t)/\theta)^2-1]$ is an initial interval for searching $\psii(t)$ for all $\nu\in(-1/2,\infty)$. 
	\par
	Kendall's tau and the coefficients of tail dependence of a GIG copula are given by 
	\begin{align*}
		\tau=\tau(\nu,\theta)=1-\int_0^\infty t(\theta h_{\nu+1,\theta}(t)/K_\nu(\theta))^2\,dt,\quad\lambda_L=\lambda_U=0,
	\end{align*}
	where $h_{\nu+k,\theta}(t)=K_{\nu+k}(\theta\sqrt{1+t})/\sqrt{1+t}^{\nu+k}$, $t\in[0,\infty)$, $\theta\in(0,\infty)$, and $\nu,k\in\IR$. For computing the tail dependence coefficients, consider $\psi^\prime(2t)/\psi^\prime(t)$ for the limits $t\downarrow0$ and $t\uparrow\infty$, and use $K_\nu(t)\approx\sqrt{\pi/(2t)}\exp(-t)$ (valid for $t\gg\lvert\nu^2-1/4\rvert$) in the latter case. A numerically stable evaluation of the integral formula for Kendall's tau for small $\nu>0$ based on the Clayton limit is given in the \textsf{R} package \texttt{nacopula}. Note that as numerical results indicate, Kendall's tau is decreasing in both $\nu$ and $\theta$ if $\nu\in[0,\infty)$; see Figure~\ref{fig.GIG} (left). If $\nu\in(0,\infty)$, the limit for Kendall's tau as $\theta\searrow0$ is $1/(1+2\nu)$ which equals Kendall's tau for Clayton's family with parameter $1/\nu$.	Figure~\ref{fig.GIG} (right) shows a scatter plot of 1000 bivariate vectors of random variates following a GIG family with $\th=(0.05,0.0968)\T$ with corresponding Kendall's tau equal to 0.5.
	\begin{figure}[htbp]
	 	\centering
	    \includegraphics[width=0.48\textwidth]{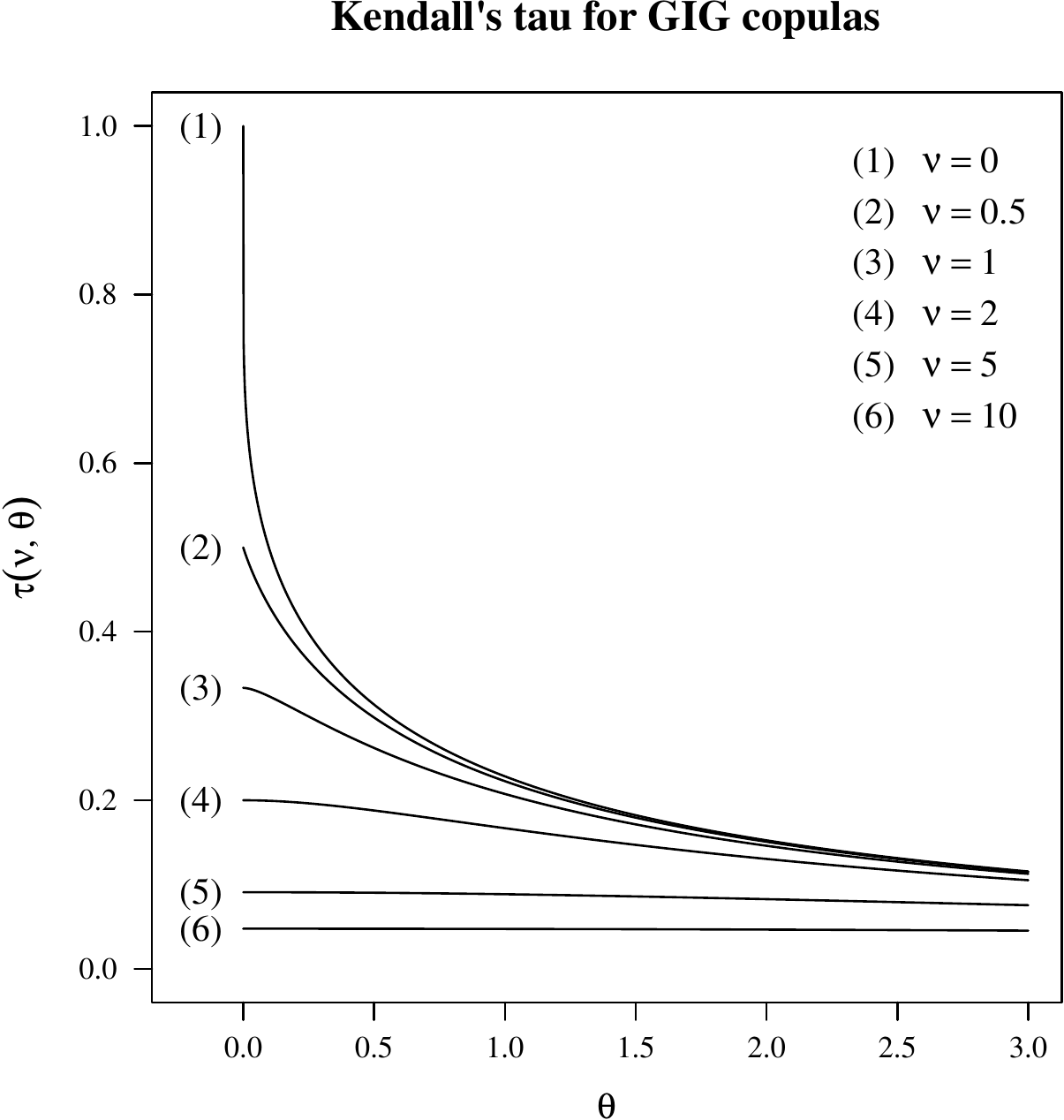}%
	    \hfill
	 	 \includegraphics[width=0.48\textwidth]{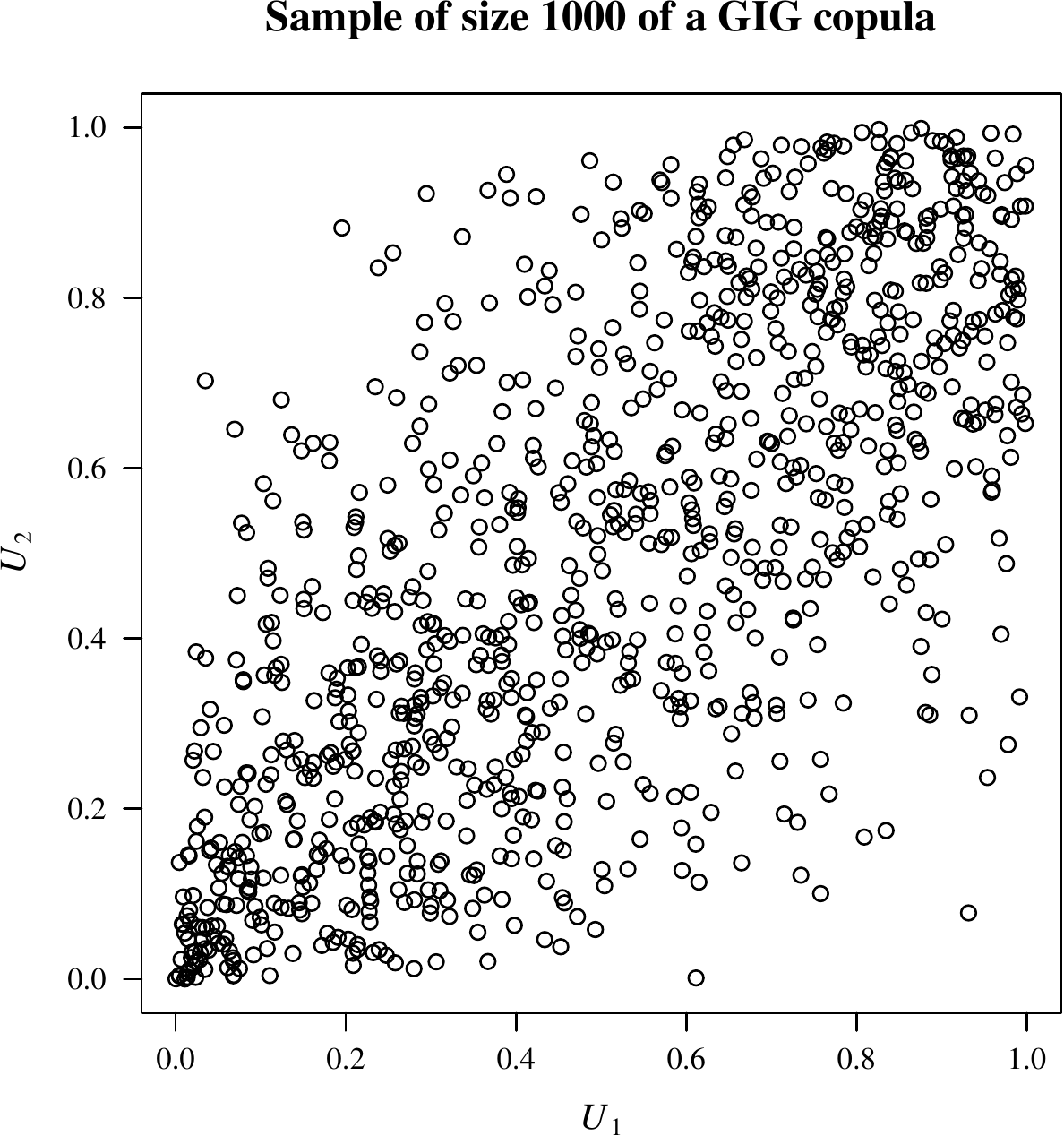}%
	 	\setcapwidth{\textwidth}%
	 	\caption{Kendall's tau for the GIG family as a function in $\theta$ for given $\nu$ (left). A sample of 1000 bivariate vectors of random variates following a GIG family with $\th=(0.05,0.0968)\T$ and corresponding Kendall's tau equal to 0.5 (right).} 
	 	\label{fig.GIG}
	\end{figure}
	\par
	One advantage of the GIG family is that the generator derivatives take on a comparably simple form, which can be represented as
	\begin{align*}
		(-1)^d\psi^{(d)}(t)=\frac{h_{\nu+d,\nu,\theta}(t)}{2^d(1+t)^{\nu+d}},\ t\in(0,\infty),\ d\in\IN_0.
	\end{align*}
	This can be easily derived by differentiating $\psi$ under the integral sign and interpreting the resulting integrand as the density of a $\GIG(\nu+d,2(1+t),\theta^2/2)$ distribution which integrates to one. Via (\ref{c}), one then easily finds the form of the log-density of a GIG family, given by
	\begin{align*}
		\log c(\bm{u})&=\log h_{\nu+d,\theta}(t_{\th}(\bm{u}))+(d-1)\log K_\nu(\theta)-\sum_{j=1}^d\log h_{\nu+1,\theta}(\psii(u_j)),\\
		              &=-(\nu+d)\log(1+t_{\th}(\bm{u}))+\log h_{\nu+d,\nu,\theta}(t_{\th}(\bm{u}))+(\nu+1)\sum_{j=1}^d\log(1+\psii(u_j))\\
		&\phantom{={}}-\sum_{j=1}^d\log h_{\nu+1,\nu,\theta}(\psii(u_j)).
	\end{align*}
	\par
	The following algorithm describes a procedure for finding an initial interval for GIG copulas with $\nu\in[0,\infty)$. The idea behind this algorithm is similarly to those of Algorithm \ref{alg.ii.increasing}. However, it takes into account that $\tau(\nu,\theta)$ is decreasing in both parameters $\nu\in[0,\infty)$ and $\theta\in(0,\infty)$ and thus first determines the upper-left, then the lower-left, and finally the lower-right endpoint of the initial rectangle. As before, $\hat{\tau}_n$ denotes an estimator of Kendall's tau, taken as the pairwise Kendall's tau estimator. Note that for $\nu=\nu_l=0$, the range of Kendall's tau as a function in $\theta$ is $(0,1]$; see Figure \ref{fig.GIG} (left). Furthermore, for $\theta=\theta_l$ sufficiently small, the range of Kendall's tau as a function in $\nu$ is $(0,1)$.
	\begin{algorithm}\label{alg.ii.decreasing}
		\myskipalgo
		\linespread{1.22}\normalfont
		\begin{tabbing}
			\hspace{8mm} \= \hspace{4mm} \= \hspace{80mm} \= \kill
			(1) \> Choose $h_-,h_+\ge0$, and $\eps>0$.\\
			(2) \> Let the smallest $\nu$ be denoted by $\nu_l=0$.\\
			(3) \> Solve $\tau(\nu_l,\theta_u)=\max\{\hat{\tau}_n-h_-,\eps\}$ with respect to $\theta_u$.\\
			(4) \> Solve $\tau(\nu_l,\theta_l)=\min\{\hat{\tau}_n+h_+,1-\eps\}$ with respect to $\theta_l$.\\
			(5) \> Solve $\tau(\nu_u,\theta_l)=\max\{\hat{\tau}_n-h_-,\eps\}$ with respect to $\nu_u$.\\
			(6) \> Return the initial interval $I=[(\nu_l,\theta_l)\T,(\nu_u,\theta_u)\T]$.\\
		\end{tabbing}
	\end{algorithm}
	\par
	We generate $N=200$ times $n=100$ realizations of i.i.d.\ random vectors in $d=10$ dimensions following GIG copulas with parameters $\th=(\nu,\theta)\T=(0.1,0.8333)\T$ resulting in a Kendall's tau of 0.25, $\th=(0.05,0.0968)\T$ with corresponding Kendall's tau equal to 0.5, and $\th=(0.01,0.0012)\T$ with Kendall's tau equal to 0.75 (the choice of $N$ is made solely to the larger run time for this family). For finding initial intervals, Algorithm \ref{alg.ii.decreasing} is applied with $\eps=0.005$ and $h_-=h_+=0.15$. The results are summarized in Table~\ref{tab.GIG}. Note that especially under weak concordance, the GIG family requires a larger sample size $n$ in order for the bias to be small. 
	\begin{table}[htbp] 
    \centering
	 \begin{tabularx}{\textwidth}{@{\extracolsep{\fill}}cd{1}{2}d{1}{2}d{1}{4}d{1}{4}d{1}{4}d{2}{4}d{1}{4}d{3}{0}d{3}{2}}
      \toprule	
      &&&\multicolumn{2}{c}{$\hat{\nu}_n$}&&\multicolumn{2}{c}{$\hat{\theta}_n$}&\\
		\cmidrule(lr{0.4em}){4-5}\cmidrule(lr{0.4em}){7-8}
		\multicolumn{1}{c}{$n$}&\multicolumn{1}{c}{$\tau$}&\multicolumn{1}{c}{$\nu$}&\multicolumn{1}{c}{Bias}&\multicolumn{1}{c}{RMSE}&\multicolumn{1}{c}{$\theta$}&\multicolumn{1}{c}{Bias}&\multicolumn{1}{c}{RMSE}&\multicolumn{1}{c}{$\#$}&\multicolumn{1}{c}{MUT}\\
      \midrule
      100 & 0.25 & 0.1000 & 0.1622 & 0.3769 & 0.8333 & -0.0178 & 0.1395 & 57 & 12.3\,\text{s} \\ 
      100 & 0.5 & 0.0500 & 0.0065 & 0.0538 & 0.0968 & 0.0002 & 0.0134 & 55 & 12.7\,\text{s} \\ 
		100 & 0.75 & 0.0100 & 0.0198 & 0.0541 & 0.0012 & -0.0001 & 0.0005 & 93 & 39.3\,\text{s} \\ \addlinespace[4pt]
		500 & 0.25 & 0.1000 & 0.0480 & 0.1680 & 0.8333 & -0.0003 & 0.0562 & 63 & 67.0\,\text{s} \\ 
      500 & 0.5 & 0.0500 & 0.0034 & 0.0295 & 0.0968 & 0.0003 & 0.0063 & 53 & 60.4\,\text{s} \\
      500 & 0.75 & 0.0100 & 0.0057 & 0.0286 & 0.0012 & -0.0000 & 0.0003 & 102 & 168.3\,\text{s} \\  
      \bottomrule
    \end{tabularx}
    \setcapwidth{\textwidth}%
    \caption{Summary statistics for estimating two-parameter GIG copulas.}
    \label{tab.GIG}
  \end{table}
\section{Conclusion}\label{sec.con}
	We presented explicit functional forms for the generator derivatives of well-known Archimedean copulas. These explicit formulas are of interest for several reasons. Apart from being able to express various important quantities such as conditional distributions or the Kendall distribution function explicitly, the generator derivatives allow us to apply maximum-likelihood estimation for estimating the parameter vectors of various Archimedean copulas, even in large dimensions such as $d=100$. The excellent performance in terms of both precision and run time of maximum likelihood estimation was shown in \textcite{hofertmaechlermcneil2011a} in a large-scale comparison with various other estimators up to dimension $d=100$. In the present work, we presented the theoretical details and showed that maximum-likelihood estimation is also feasible for multi-parameter Archimedean families. Furthermore, we showed that the mean squared error MSE is decreasing in the dimension $d$ and that this decrease is of the same order as the decrease in the sample size $n$, that is, $\text{MSE}\propto1/(nd)$. We also constructed initial intervals for the likelihood optimization. Moreover, we obtained likelihood-based confidence intervals for the parameter vector and compared them to information-based confidence intervals for the Clayton family where the Fisher information is comparably easy to compute. A transparent implementation of the presented results is given in the open source \textsf{R} package \texttt{nacopula}, so that the interested reader can easily follow our calculations.  
\subsection*{Acknowledgements}
	The authors would like to thank Khristo Boyadzhiev (Ohio Northern University) for introducing us to and guiding us through the fascinating world of exponential polynomials.
\bibliographystyle{plainnat}
\bibliography{./mybib} 
\end{document}